\documentclass[a4paper, 11pt]{article}

\input xy
\xyoption{all}

\usepackage{amsmath,amsthm,amssymb}
\usepackage[page,header]{appendix}
\usepackage{morefloats}
\usepackage{color}
\usepackage{makeidx}
\usepackage{hyperref}	 	
\usepackage{fancybox}

\makeindex 

\usepackage{latexsym, amsmath, amssymb, amsfonts, amscd}
\usepackage{amsthm}
\usepackage{t1enc}
\usepackage[mathscr]{eucal}
\usepackage{indentfirst}
\usepackage{graphicx, pb-diagram}
\usepackage{enumerate}
\usepackage[all,poly,web,knot]{xy}

\newcommand{\Title}{Title}

\numberwithin{equation}{section}

{\theoremstyle{definition}\newtheorem{definition}{Definition}[section]

\newtheorem{defititle}[definition]{\Title}

\newtheorem{remark}[definition]{Remark}
\newtheorem{remarks}[definition]{Remarks}
\newtheorem{ep}[definition]{Example}
\newtheorem{eps}[definition]{Examples}}
\newtheorem{prop}[definition]{Proposition}
\newtheorem{proposition-definition}[definition]{Proposition-Definition}
\newtheorem{lemma}[definition]{Lemma}
\newtheorem{thm}[definition]{Theorem}
\newtheorem{cor}[definition]{Corollary}

\newcommand{\cG}{\mathcal{G}}

\newcommand{\cF}{\mathcal{F}}
\newcommand{\cE}{\mathcal{E}}

\newcommand{\cU}{\mathcal{U}}

\newcommand{\cH}{\mathcal{H}}
\newcommand{\cM}{\mathcal{M}}

\newcommand{\cf}{{\it cf.}\/ }

\newcommand{\vX}{\mathfrak{X}}

\def\gpd{\,\lower1pt\hbox{$\longrightarrow$}\hskip-.24in\raise2pt
             \hbox{$\longrightarrow$}\,}

\renewcommand{\latticebody}{\drop@{ }}

\newcommand{\N}{\ensuremath{\mathbb N}}

\newcommand{\R}{\ensuremath{\mathbb R}}

\newcommand{\g}{\ensuremath{\mathfrak{g}}}

\newcommand{\cL}{{\mathcal L}}
\newcommand{\cA}{\mathcal{A}}
\newcommand{\cC}{\mathcal{C}}            
\newcommand{\cN}{\mathcal{N}}

\newcommand{\bt}{\mathbf{t}}                  
\newcommand{\bs}{\mathbf{s}}                  

\newcommand{\un}{\underline}

\def\act{\mathbin{\hbox{$<\kern-.4em\mapstochar\kern.4em$}}}
\def\ract{\mathbin{\hbox{$\mapstochar\kern-.3em>$}}}

\def\PB(#1,#2,#3,#4){\left\{\begin{matrix}#1&\!\!\!\stackrel{?}{\longrightarrow}&\!\!\!#2\\
\downarrow&&\!\!\!\downarrow\\
#3&\!\!\!\stackrel{?}{\longrightarrow}&\!\!\!#4\end{matrix}\right\}}

\def\pb(#1,#2,#3,#4){ \hom(#1 \to #3, #2 \to #4)}

\begin{document}

\begin{center}
{\Large\bf Holonomy transformations   for singular foliations 

\bigskip

{\sc by Iakovos Androulidakis and Marco Zambon}
}
 
\end{center}

{\footnotesize
National and Kapodistrian University of Athens
\vskip -4pt Department of Mathematics
\vskip -4pt Panepistimiopolis
\vskip -4pt GR-15784 Athens, Greece
\vskip -4pt e-mail: \texttt{iandroul@math.uoa.gr}
  
\vskip 2pt Universidad Aut\'onoma de Madrid
\vskip-4pt Departamento de Matem\'aticas 
\vskip-4pt and ICMAT(CSIC-UAM-UC3M-UCM)
\vskip-4pt Campus de Cantoblanco, 28049 - Madrid, Spain
\vskip-4pt e-mail: \texttt{marco.zambon@uam.es, marco.zambon@icmat.es}
}
\bigskip
\everymath={\displaystyle}

\date{today}

\begin{abstract}\noindent 
In order to understand   the linearization problem around a leaf of a singular foliation, we extend the familiar holonomy map from the case of  regular foliations to the case of singular foliations. To this aim we introduce the notion of holonomy transformation. Unlike the regular case, holonomy transformations can not be attached to classes of paths in the foliation, but rather to elements of the holonomy groupoid of the singular foliation.   
 
Holonomy transformations allow us to link the linearization problem with the compactness of the isotropy group of the holonomy groupoid, as well as with the linearization problem for proper Lie groupoids.  We also  study the  deformations of a singular foliation, recovering the deformation cocycles of Crainic--Moerdijk as well as those of Heitsch.
\end{abstract}
 
\setcounter{tocdepth}{2} 
\tableofcontents

\section*{Introduction}
\addcontentsline{toc}{section}{Introduction}

\subsection*{Historical overview and motivations}

A great deal of   foliation theory is based on the understanding of the action of the holonomy pseudogroup on the transversal structure of the foliation. Geometrically, the holonomy of a 
(regular) foliation $(M,F)$   at a point $x\in M$ is realised by a map $h_x \colon \pi_1(L) \to GermDiff(S)$, where $L$ is the leaf at $x$, $S$ is a transversal at $x$, and $GermDiff(S)$ is the space of germs of local diffeomorphisms of $S$. Its linearisation $Lin(h_x) \colon  \pi_1(L) \to GL(N_xL)$ is a representation on the normal space to $L$ at $x$. When one considers all pairs of points in leaves of $M$, the linearization gives rise to a representation $Lin(h)$ of the holonomy groupoid on  $TM/F$, the normal bundle to the leaves. Notice that $TM/F$ plays the role of the tangent bundle of the quotient space $M/F$ (\cf \cite[\S 10.2]{Connes}), and is the starting point for various invariants carrying geometric, topological, and analytic information of the given foliation (an account of which was given in \cite{AnZa11}).

As in \cite{AnZa11}, in the current article we are concerned with the much larger class of \textit{singular} foliations.   Recall that we understand a (singular)  foliation on a smooth manifold $M$ as a $C^{\infty}(M)$-submodule $\cF$ of compactly supported vector fields on $M$ which is locally finitely generated and closed by the Lie bracket. We extend the notion of holonomy to the singular case, in order to 
 understand the linearization and stability properties of singular foliations. The results we wish to ultimately generalize to the singular case are the following:
\begin{itemize}
\item The \textit{local Reeb stability theorem}, which 
 gives a normal form of a (regular) foliated manifold $(M,F)$ around a compact leaf $L$. Namely, when the leaf $L$ has finite holonomy, the theorem states that there is a neighborhood of $L$ which is diffeomorphic to $\frac{\widetilde{L} \times N_xL}{\pi_1(L)}$ endowed with the ``horizontal'' foliation (see \cite[Thm 2.9]{MM}). Here the representation of $\pi_1(L)$ on $N_xL$ is exactly $Lin(h_x)$. This quotient is diffeomorphic to 
$NL$,   the normal bundle  of the leaf, endowed with the linearization of the foliation $F$. Hence the local Reeb stability theorem can be viewed as a linearization result.

\item  A certain \emph{cocycle} defined by  Heitsch \cite{Heitsch}, which controls deformations of foliations.
 Note that Crainic and Moerdijk, studying deformations of Lie algebroids in \cite{CMdef},  introduce a cohomology theory (deformation cohomology) and a certain cocycle which controls such more general deformations; they also show that it recovers Heitsch's cocycle when the Lie algebroid is a regular foliation.  

\item 
The notion of \emph{Riemannian foliation}. A Riemannian foliation consists of a regular foliation $F$ on a Riemannian manifold $M$ such that  such that the action $Lin(h)$ of the holonomy groupoid on the normal bundle $F^{\perp}\to M$ is by isometries \cite{Mo88}\cite[\S 1]{Hur}.   A lot can be said about the structure and topology of  Riemannian foliations, see for instance \cite{Mo88}\cite{MM}.
\end{itemize}

We elaborate on the singular version of the first item above, that is, 
 the question of whether a singular foliation is isomorphic to its linearization. This question is interesting already in the case of singular foliations generated by one vector field. In this case, it reads as follows and is an interesting problem in the theory of differential equations:
\begin{itemize}
\item Consider a vector field $X$ on a smooth manifold $M$ vanishing at a point $x$. Its linearization   is the vector field $X_{lin}$ on $T_xM$ defined by the first-order (linear) term of the Taylor expansion of $X$ at $x$. Under what  assumptions  is there  a diffeomorphism $\phi$ from a neighborhood of $x$ to neighborhood $V$ of the origin  in $T_xM$ and a nowhere-vanishing function $f\in C^{\infty}(V)$ such that $\phi$ identifies $X$ with $f\cdot X_{lin}$? When this occurs, $X$ and $X_{lin}$ are said to be \emph{orbitally equivalent}, as their orbits are identified by $\phi$. The literature\footnote{The question of whether there is a diffeomorphism mapping $X$ straight to $X_{lin}$ appears to be treated more in the literature, see Sternberg's \cite[Thm. 4]{Ste57} and, in the setting of homeomorphisms, the Hartman-Grobman Theorem \cite[\S 2.8]{Pe01}.}
  seems to provide an answer to this question only when $M$ has dimension $2$ and in the formal setting, see \cite[Prop. 4.29]{IY08}.
\end{itemize}

Recent work by Crainic and Struchiner \cite{CrStr} on the linearization of proper Lie groupoids does provide a linearization result for those singular foliations $\cF$ which arise from such groupoids. One question that arises naturally is  what role  the holonomy groups (namely the isotropy groups of the holonomy groupoid) play in the linearizability of the foliation. After all, much like regular foliations and the Reeb stability theorem there, also in the singular case it is reasonable to require linearization conditions  using the least possible information, and the holonomy groupoid naturally provides the correct framework for this among all other groupoids the foliation is possibly defined from. However, techniques which work in the smooth category, like the ones developed in \cite{CrStr}, cannot be applied to the holonomy groupoid of a singular foliation, given the pathology of its topology (see \cite{AndrSk}).  In fact, the range of this applicability in the framework of singular foliations is an endeavour of different order, well worth investigating in a separate article.

\subsection*{Main tools and overview of results}
Our current study relies once again on the construction of the holonomy groupoid $H$ for \textit{any} singular foliation $(M,\cF)$ given in \cite{AndrSk}, and the notion of bi-submersion introduced there. Recall that $H$ is a topological groupoid associated to $(M,\cF)$ (in particular, it encodes more information  than just  the partition of $M$ into leaves of $\cF$), and that  a bi-submersion is a smooth cover of an open subset of the (often) topologically pathological groupoid $H$.  {In \cite{Debord} and \cite{AnZa11} it was shown that the restriction of $H$ to a leaf is smooth; we build strongly on this result.}  

If one tries to define the holonomy of a singular foliation starting from (classes of) paths, as in the regular case, one obtains a very coarse notion, which does not allow for linearization. To remedy this, we introduce the notion of  \emph{holonomy transformation}, an equivalence class of germs of diffeomorphisms. It is the correct replacement of $GermDiff(S)$ for a singular foliation, as it encodes the geometric idea of holonomy and specializes correctly in the regular case. We explain this in the first item of the following list, which presents our main results in \S \ref{section:geomhol}--\S \ref{sec:lht}:

\begin{itemize}
 
\item For $x, y$ in the same leaf of a singular foliation $(M,\cF)$, consider transversal slices $S_x, S_y$ to the leaf at $x$ and $y$ respectively. There is a well-defined map $$\Phi_x^y \colon H_x^y \to \frac{GermAut_{\cF}(S_x;S_y)}{exp(I_x \cF_{S_x}) },$$ suitably constructed restricting the flows of vector fields in $\cF$.
The above target space is the space of germs of foliation-preserving local diffeomorphisms between $S_x$ and $S_y$, quotiented by the exponentials of elements in the maximal ideal $I_x\cF_{S_x}$ of the restriction of $\cF$ to $S_x$ (Thm. \ref{globalaction}).  Elements of the target are, by definition, holonomy transformations.

 The maps $\Phi_x^y$ assemble to a morphism of groupoids $\Phi \colon H\to\{\text{holonomy transformations}\}$, which in the case of regular foliations recovers the usual notion of holonomy given assembling  the maps $h_x$ introduced earlier. We prove that the map $\Phi$ is injective (Thm. \ref{thm:inj}), and therefore the holonomy groupoid $H$ obtains the following geometric characterization: it can be viewed as a subset of the set of holonomy transformations.
 
\item The map $\Phi$ linearizes to a morphism of groupoids $Lin(\Phi) \colon H \to Iso(N)$, whose target is the groupoid of isomorphisms between the fibres of the (singular) normal bundle to the leaves $N$ (linear holonomy). See Prop. \ref{globallinaction}.

\item Although the normal ``bundle'' $N$ is a singular space, its sections form a nice $C^{\infty}(M)$-module $\cN = \vX(M)/\widehat{\cF}$, where $\widehat{\cF}$ denotes a canonical completion of $\cF$. The ``bundle'' $N$ carries transversal information only up to first order, while all the higher order transversal data is carried by the module $\cN$. We show that $Lin(\Phi)$ differentiates to the ``Bott connection'' $\widehat{\cF} \times \cN \to \cN$ (Prop. \ref{transvrep}).
\end{itemize}

The map $\Phi$ above  encodes the geometric notion of holonomy for singular foliations -- quite a non-trivial notion --, and as such it is geometrically relevant and interesting. The well-definess of $\Phi$ is the technically hardest result in this paper, and 
the existence of  $\Phi$ gives a full geometric justification for the terminology ``holonomy groupoid''.
Our results toward a generalization of the Reeb stability theorem are based   on its linearization $Lin(\Phi)$. More precisely, they are based on 
the fact that, given a leaf satisfying certain regularity conditions, the restriction of  $Lin(\Phi)$ to the leaf is a 
\emph{Lie groupoid representation.} 
 The main results of \S \ref{section:linfol} are:
\begin{itemize}

\item We give two local models for the foliation around a   leaf $L$:
\begin{enumerate}
\item The normal bundle $NL$ is endowed with the foliation $\cF_{lin}$  generated by the linearizations of vector fields in $\cF$.
Under regularity conditions on $L$, $\cF_{lin}$ is the foliation induced by the above-mentioned Lie groupoid representation.
\item Under regularity conditions on the leaf $L$,  the quotient $Q=\frac{H_x \times N_xL}{H_x^x}$ is smooth,   where $H_x^x$ acts on $N_xL$ by the restriction of the linear holonomy $Lin(\Phi)$.
$Q$ is endowed with a canonical singular foliation.

\end{enumerate}
We show that the two models are isomorphic (see Prop. \ref{prop:equiv2m}). Notice that the second model $Q$ is the natural generalization of the model appearing in the Reeb stability theorem for regular foliations.
 
\item Under regularity conditions on the leaf $L$
we show the following equivalence, where $x\in L$:    $\cF$ is linearizable about $L$ and $H_x^x$ is compact if{f} $\cF$ (in a neighborhood of $L$) is the foliation induced by a Hausdorff Lie groupoid which is proper at $x$ (see Prop. \ref{prop:proper}).

In this case, $\cF$ admits the structure of a singular Riemannian foliation  around $L$ (Prop. \ref{prop:Riem}).
\end{itemize}

A key ingredient to prove the last item above is the following observation \cite[Ex. 3.4(4)]{AndrSk}: when a foliation $\cF$ is induced by a Lie groupoid $G$, then the holonomy groupoid $H$ of $\cF$ is a quotient of $G$ (in particular, the properness of $G$ at $x$ implies the compactness of $H_x^x$). This observation implies   that there is a huge class of linearizable foliations which \textit{cannot} be defined by any proper Lie groupoid, namely those foliations which are defined from a linear action of a non-compact group (see remark \ref{rem:linearization}). Further, the problem of when a given Lie algebroid $A$ integrates to a proper
Lie groupoid is an open  one, and the above observation shows: if the foliation defined by
$A$ has non-compact isotropy group $H_x^x$ at some point $x$, then there is no
proper Lie groupoid integrating $A$.

Finally in \S \ref{section:deform} we consider deformations:
\begin{itemize} 

\item  We define the cohomology groups $H^*_{def}(\widehat{\cF})$ (deformation cohomology) and, using the ``Bott connection'', we define $H^*(\widehat{\cF},\cN)$ (foliated cohomology). Using the techniques of  \cite{CMdef} we find  that deformations of $(M,\cF)$ with isomorphic underlying $C^{\infty}(M)$-module structure are controlled by a certain element of $H^2_{def}(\widehat{\cF})$. There exists a canonical map $H^2_{def}(\widehat{\cF}) \to H^1(\widehat{\cF},\cN)$, and in the regular case the image of the above cocycle is exactly Heitsch's class.
\end{itemize}
 
\noindent\textbf{Notation:} Given a manifold $M$, we  use $\vX(M)$ to denote its vector fields, and  $\vX_c(M)$ its vector fields with compact support. For a vector field $X$ and $x\in M$, we use $exp_x(X)\in M$ to denote the time-one flow of $X$ applied to $x$. By $\cF$ we will always denote a singular foliation on $M$ and $L$ a leaf. If $X \in \cF$, we use $[X]$ to denote the class $X \text{ mod }I_x\cF$ (here $x\in M$). The notation $\langle X \rangle$ is used to denote either classes under several other equivalence relations or  the foliation generated by $X$.\\
The holonomy groupoid of a singular foliation is denoted by $H$. Further, $H_x=\bs^{-1}(x)$ is the source fiber  and $H_x^x=\bs^{-1}(x)\cap \bt^{-1}(x)$ the isotropy group at $x$; the same notation applies to bi-submersions  $U,W,...$

\noindent\textbf{Acknowledgements:} We would like to thank Marius Crainic,  Rui Fernandes, Camille Laurent, Kirill Mackenzie, Ioan Marcu\c{t}, Ivan Struchiner and especially Georges Skandalis for illuminating discussions. 
We thank the anonymous referee for  suggestions towards improving the presentation of the material.

I. Androulidakis was partially supported by a Marie Curie Career Integration Grant (FP7-PEOPLE-2011-CIG, Grant No PCI09-GA-2011-290823), by DFG (Germany) through project DFG-Az:Me 3248/1-1 and by FCT (Portugal) with European Regional Development Fund (COMPETE) and national funds through the project PTDC/MAT/098770/2008. He also acknowledges the support of a Severo Ochoa grant during a visit to ICMAT. 

M. Zambon was partially supported by CMUP and FCT (Portugal) through the programs POCTI, POSI and Ciencia 2007, by projects PTDC/MAT/098770/2008 and PTDC/MAT/099880/2008 (Portugal) and by  projects MICINN RYC-2009-04065,  MTM2009-08166-E, MTM2011-22612 and  MINECO: ICMAT Severo Ochoa project SEV-2011-0087(Spain). He thanks Universiteit Utrecht and  Universit\"at G\"ottingen for  hospitality, and Daniel Peralta-Salas for advice and references.

\section{Background material} 

For the convenience of the reader we give here an outline of the constructions and results of \cite{AndrSk} and \cite{AnZa11}. 

\subsection{Foliations}\label{sec:fol}

Let $M$ be a smooth manifold. Given a vector bundle $E \to M$ we denote $C^{\infty}_c(M;E)$ the $C^{\infty}(M)$-module of compactly supported sections of $E$. Stefan \cite{Stefan} and Sussmann \cite{Sussmann} showed that the following definition gives rise to a partition of $M$ into injectively immersed submanifolds (leaves):
\begin{definition}\label{def:fol}
A \textit{(singular) foliation} on $M$ is a locally finitely generated submodule $\cF$ of the $C^{\infty}(M)$-module $\vX_c(M) = C^{\infty}_c(M;TM)$, stable by the Lie bracket.
\end{definition}
It was shown in \cite{Stefan, Sussmann} that such a module induces a partition of $M$ to (immersed) submanifolds, called leaves. The leaf at $x \in M$ of a singular foliation $\cF$ is the set of points in $M$ which can be connected to $x$ following integral curves of vector fields in $\cF$.

In general, a singular foliation contains more information than the underlying   partition of $M$ into leaves. Singular foliations arise in many natural geometric contexts: from actions of Lie groups and, more generally, from Lie groupoids and Lie algebroids.
The following apparatus is naturally associated with a foliation, and will be of use   in the current article.
\begin{enumerate}
\item[a)] A leaf $L$ is \textit{regular} if there exists an open neighborhood $W$ of $L$ in $M$ such that the dimension of $L$ is equal to  the dimension of any other leaf intersecting $W$. Otherwise $L$ will be called \textit{singular}.

\item[b)] For $x \in M$ let $I_x = \{f \in C^{\infty}(M) : f(x) = 0\}$ and consider the maximal ideal $I_x\cF$ of $\cF$. Since $\cF$ is locally finitely generated, the quotient $\cF_x = \cF/I_x\cF$ is a finite dimensional vector space. Let $F_x$ be the tangent space at $x$ of the leaf $L_x$ passing through $x$. We have a short exact sequence of vector spaces $$0 \to \g_x \to \cF_x \stackrel{ev_x}{\to} F_x \to 0$$ Its kernel $\g_x$ is a Lie algebra, which vanishes iff $L_x$ is a regular leaf.
Explicitly, $\g_x=\cF(x)/I_x\cF$ where $\cF(x):=\{X\in \cF:X_x=0\}$.

\item[c)] Let $L$ be a leaf and put $I_L$ the space of smooth functions on $M$ which vanish on $L$. Then $A_L = \cup_{x \in L}\cF_x$ is a transitive Lie algebroid over $L$. If $L$ is embedded then $C^{\infty}_c(L;A_L) = \cF/I_L\cF$. If $L$ is immersed then the previous equality holds locally (see \cite[Rem. 1.8]{AnZa11}).

 \item[d)]   We can pull back a foliation along a {smooth map}: If $p: N \longrightarrow M$ {is smooth} then $p^{-1}(\cF)$ is the submodule 
of   $\vX_c(N)$ consisting of 
$C^{\infty}_{c}(N)$-linear combinations of vector fields on $N$ which are projectable and project to elements of $\cF$.
\end{enumerate}

Associated to $\cF$ there is a canonical  natural submodule,  which contains non-compactly supported vector fields  when $M$ is not compact (see \cite[\S 1.1]{AndrSk}):
\begin{definition}\label{def:folhat} The \emph{completion} of $\cF$ is the following $C^{\infty}(M)$-submodule of $\vX(M)$:
$$\widehat{\cF}:=\{X\in \vX(M): fX\in \cF \text{ for all } f\in  C^{\infty}_c(M)\}.$$ 
\end{definition}
Associating to each open subset $U$ of $M$ the $C^{\infty}(U)$-module $\widehat{\cF|_U}$ we obtain a sheaf. Indeed, this is the sheafification of the presheaf which associates $\{X|_U:X\in \cF\}$
to $U$.

Passing from $\cF$ to $\widehat{\cF}$ we do not lose any information, since $\cF$ is recovered as the set of compactly supported elements of $\widehat{\cF}$.
Indeed there is a bijection  
$$ \text{\{submodules of $\vX_c(M)$\}}\to  \{\text{submodules of $\vX(M)$ giving rise to sheaves\}}, \;\;\cE \mapsto \widehat{\cE},$$
whose inverse map takes $\cG$ to its submodule of compactly supported sections (it is generated by $\{fY: f\in C_c^{\infty}(M), Y \in \cG\}$.) Further, $\cE$ is locally finitely generated if{f} $\widehat{\cE}$ is, and the same holds for involutivity.  {One checks easily that for every $x \in M$, the vector spaces $\cF_x, \g_x$ coincide with $\widehat{\cF}/I_x\widehat{\cF}$ and $\widehat{\cF}(x)/I_x\widehat{\cF}$ respectively.} We will make use of $\widehat{\cF}$ only in \S\ref{subsection:linholrep}, \S\ref{section:deform} and \S\ref{subsec:normalmodule}.

 The local picture of a foliation is the following: 
 \begin{prop}\label{thm:splitting}
{\em\bf{(Splitting theorem)}} Let $(M,\cF)$ be a manifold with a foliation and $x \in M$, and set $k := dim (F_{x})$. Let $\hat{S} $ be  a slice at $x$, that is, an embedded submanifold such that $T_x\hat{S}\oplus F_x=T_xM$.

 Then there exists an open neighborhood $W$ of $x$ in $M$ and a diffeomorphism of foliated  manifolds 
 \begin{equation}\label{diff}
(W,\cF_W)\cong (I^k,TI^k) \times (S,\cF_S). 
\end{equation}
 Here {$\cF_W$ is the restriction of $\cF$ to $W$,} $I:=(-1,1)$, $S:=\hat{S} \cap W$ and   $\cF_S$ consists of the restriction to $S$ of vector fields in $W$ tangent to $S$.
 
In particular, if  we denote by $s_1,\dots,s_k$ the canonical coordinates on $I^k$ and $X_1,\dots,X_l$ are generators of $\cF_S$, then
$\cF_W$ is generated by $\partial_{s_1},\dots, \partial_{s_k}$ and the (trivial extensions of) $X_1,\dots,X_l$.
\end{prop}

\subsection{Holonomy groupoids}\label{section:holgpd}

Let $(M,\cF)$ be a (singular) foliation. We recall the notion of bi-submersion from \cite{AndrSk} and the construction of the holonomy groupoid.
\begin{enumerate}
    \item A {\em{bi-submersion}} of $(M,\cF)$ is a smooth manifold $U$ endowed with two submersions $\bt, \bs : U \longrightarrow M$ satisfying:
        \begin{enumerate}
            \item[(i)] $\bs^{-1}(\cF) = \bt^{-1}(\cF)$,
            \item[(ii)] $\bs^{-1}(\cF) =  C^{\infty}_{c}(U;\ker d\bs) + C^{\infty}_{c}(U;\ker d\bt)$.
        \end{enumerate}
We say $(U,\bt,\bs)$ is {\em{minimal}} at $u$ if $dim(U) = dim(M) + dim(\cF_{\bs(u)})$.
    \item\label{phbisubm} Let $x\in M$, and $X_1,\ldots,X_n \in \cF$ inducing a basis of $\cF_x$.  In \cite[Prop. 2.10 a)]{AndrSk} it was shown that there is an open neighborhood $U$ of $(x,0)$ in $M \times \R^n$ such that $(U,\bt_U,\bs_U)$ is a bi-submersion minimal at $(x,0)$, where $\bs_U(y,\xi)=y$ and
   $\bt_U(y,\xi) = exp_y(\sum_{i = 1}^n \xi_i X_i)$. (Recall that the latter is the image of $y$ under the time-$1$ flow of $\sum_{i = 1}^n \xi_i X_i$.) Bi-submersions arising this way are called  {\em{path holonomy bi-submersions}}.  

    \item Let $(U_i,\bt_i,\bs_i)$ be bi-submersions, $i = 1,2$. Then $(U_i,\bs_i,\bt_i)$ are bi-submersions, as well as $(U_1\circ U_2,\bt,\bs)$ where $U_1\circ U_2 = U_1 \times_{\bs_1,\bt_2}U_2$, $\bt(u_1,u_2) = \bt(u_1)$ and $\bs(u_1,u_2) = \bs(u_2)$. They are called the \emph{inverse} and \emph{composite} bi-submersions respectively.

	\item Let $(U,\bt_{U},\bs_{U})$ and $(V,\bt_{V},\bs_{V})$ be two bi-submersions. A {\em{morphism of bi-submersions}} is a smooth map $f : U \longrightarrow V$ such that $\bs_{V} \circ f = \bs_{U}$ and $\bt_{V} \circ f = \bt_{U}$.
    \item A {\em{bisection}} of $(U,\bt,\bs)$ is a locally closed submanifold $V$ of $U$ on which the restrictions of $\bs$ and $\bt$ are diffeomorphisms to open subsets of $M$.
     \item We say that $u \in U$ {\em{carries}} the foliation-preserving local diffeomorphism $\psi$ if there is a bisection $V$ such that $u \in V$ and $\psi = \bt\mid_V \circ (\bs\mid_V)^{-1}$.
     \item It was shown in  \cite[Cor. 2.11(b)]{AndrSk} that if $\{(U_i,\bt_i,\bs_i)\}_{i\in I}$ are bi-submersions, $i = 1,2$ then $u_1 \in U_1$ and $u_2\in U_2$  carry the same local diffeomorphism iff there exists a morphism of bi-submersions $g$ defined in an open neighborhood of $u_1 \in U_1$ such that $g(u_1) = u_2$. Such a morphism maps every bisection $V$ of $U_1$ at $u_1$ to a bisection $g(V)$ of $U_2$ at $u_2$.
\end{enumerate}

Bi-submersions are the key for the construction of the holonomy groupoid. Let us recall this construction: 
 
Given a foliation $(M,\cF)$, take a family of path holonomy bi-submersions $\{U_i\}_{i\in I}$ covering $M$, i.e. $\cup_{i\in I}\bs(U_i)=M$.
Let  $\cU$ be the family of all  finite compositions of elements of $\{U_i\}_{i\in I}$ and of  their inverses ($\cU$ is a \emph{path-holonomy atlas}, see \cite[Ex. 3.4(3)]{AndrSk}). 
The   \textit{holonomy groupoid} of the foliation $\cF$ is
the quotient 
$$H(\cF) :=\coprod_{U\in \cU}U/\sim$$ by the equivalence relation   for which $u\in U$ is equivalent to $u'\in U'$ if there is a morphism of bi-submersions $f:W\to U'$ defined in a neighborhood $W\subset U$ of $u$ such that $f(u)=u'$.
 
We denote the holonomy groupoid by $H$ when the choice of $\cF$ is clear.   Its restriction to a leaf $L$ is $H_L = \bs^{-1}(L) = \bt^{-1}(L)$. Its $\bs$-fiber is $H_x = \bs^{-1}(x)$ and its isotropy group $H_x^x = \bs^{-1}(x)\cap \bt^{-1}(x)$, where $x \in L$. The following was proven in \cite{Debord}:

\begin{thm}
The $\bs$-fibers of $H$ are smooth manifolds.
\end{thm}

In \cite{AnZa11} it was shown that a consequence of this is the following:

\begin{cor}\label{ALintegr}
The transitive groupoid $H_L$ is smooth and integrates the Lie algebroid $A_L   = \cup_{x\in L}\cF_x$.
\end{cor}

\subsection{The holonomy groups}\label{sec:essiso}

In \S \ref{section:adjoint} we will also need the following material from \cite{AnZa11}:

Pick a point $x \in L$ and consider the connected and simply connected Lie group $G_x$ integrating the Lie algebra $\g_x$. Let $\{X_i\}_{i\le n} \in \cF$  be vector fields whose images in $\cF_x$ form a basis of $\cF_x$  and such that the images of $\{X_i\}_{i\le \ell}$ form a basis of $\g_x$. Let $U$ be the corresponding path holonomy bi-submersion. Choose $\bt$-lifts $Y_i \in C^{\infty}(U;\ker d\bs)$ of the $X_i$. We can find a small neighborhood $\widetilde{G}_x$ of the identity in $G_x$ where the map $$\Delta \colon \tilde{G}_x \to U_x^x\;,\; exp_{\g_x}(\sum_{i=1}^l k_i[X_i]) \mapsto   exp_{(x,0)}(\sum_{i=1}^l k_i Y_i)$$ is a diffeomorphism onto its image. It turns out that the composition $\widetilde{\varepsilon}_x = \sharp \circ \Delta : \widetilde{G}_x \to H_x^x$ is independent of the choice of path holonomy bi-submersion and extends to a morphism of topological groups $$\varepsilon : G_x \to H_x^x$$

\section{Holonomy transformations}\label{section:geomhol}

In this section we extend the familiar notion of holonomy of a regular foliation to the singular case. 
We introduce and motivate the notion of holonomy transformation in \S 
\ref{subsec:sing}. In \S \ref{subsec:holtr}
we show how to associate holonomy transformations to a foliation: we obtain a map $\Phi \colon H \to \{\text{holonomy transformations}\}$,
generalizing what happens in the regular case. In \S \ref{sec:inj} we
show that this map is injective.

 \subsection{Overview of holonomy in the regular case}\label{subsec:reg}

A regular foliation of $M$ is given by an involutive subbundle of $TM$. Taking $\cF$ to be its module of  compactly supported sections   we obtain a foliation in the sense of Def. \ref{def:fol}.
One way to see the classical notion of holonomy of a regular foliation is as follows: Consider a path $\gamma$ from $x$ to $y$ lying in a leaf $L$ and fix transversals $S_{x}$ at $x$ and $S_y$ at $y$. The map $\gamma \colon [0,1] \to M$ can be extended to a continuous map $$\Gamma \colon S_x \times [0,1]  \to M$$ with $\Gamma|_{S_x \times \{0\}}=Id_{S_x}$, $\Gamma({S_x \times \{1\}})\subset S_y$ and such that $t \mapsto \Gamma (\tilde{x},t)$ is a curve in a leaf of $\cF$ for all $\tilde{x}\in S_x$. In particular, $\Gamma(x,t) = \gamma(t)$. The \emph{holonomy of $\gamma$} is then defined as the germ at $x$ of the map $$hol_{\gamma} : S_x \to S_y, \tilde{x} \mapsto \Gamma(\tilde{x},1).$$ We have:  
\begin{itemize}
 \item 
The  holonomy of $\gamma$ 
is independent of the choice of extension $\Gamma$, and  depends on  the homotopy class of $\gamma$ in $L$ rather than on $\gamma$ itself.
This gives rise to a map\\
 $hol \colon \{\text{homotopy classes of paths in $L$ from $x$ to $y$}\}
 \to GermAut_{\cF}(S_x;S_y)$.
 \item The holonomy of $\gamma$ can be linearised taking the derivative $d_xhol_{\gamma} \colon T_x S_x \to T_y S_y$. This gives rise to a map\\ $ \{\text{homotopy classes of paths in $L$ from $x$ to $y$}\} \to Iso(T_x S_x , T_y S_y)$ (linear holonomy).
\end{itemize}

\begin{remark}
The choice of transversal is immaterial: if $S'_x$ is another transversal at $x$, one obtains a canonical identification $S_x\cong S'_x$ near $x$ modifying slightly the above holonomy construction. If further we take a transversal $S'_y$ at $y$, then the diffeomorphism $S'_x \to S'_y$ obtained as the holonomy of $\gamma$ and the diffeomorphism $S_x \to S_y$ obtained above coincide upon applying the identifications.  
\end{remark}

\subsection{The singular case}\label{subsec:sing}

Let us consider a singular foliation $(M,\cF)$.  We make a first attempt to define a notion of holonomy, by
the following recipe  which makes use of the module $\cF$ rather than  just of the underlying partition of $M$ into leaves, and which clearly reduces to the notion of holonomy in the regular case. Our attempt will not be completely successful, but it is useful in that it motivates the definition of holonomy transformation.

Let $\gamma \colon [0,1] \to M$  be a curve from $x$ to $y$ lying in a leaf of $\cF$, and fix slices $S_x$ and $S_y$.
For every $t$ extend $\dot{\gamma}(t)$, the velocity of the curve at time $t$, to a vector field $Z^t$ lying in $\cF$, with the property that   $\Gamma \colon S_x \times [0,1]  \to M$ --  defined following the flow of  the time-dependent vector field $\{Z^t\}_{t\in[0,1]}$ starting at points of $S_x$ -- takes $S_x$ to $S_y$.
 Unlike
 the regular case, the resulting map  $S_x \to S_y, \tilde{x} \mapsto \Gamma(\tilde{x},1)$ depends on the choice of extension $\Gamma$. This can be seen looking at simple examples:
 
\begin{eps}\label{holFx}
\begin{enumerate}
\item Let $M=\R$, and $ {\widehat{\cF}}=\langle z\partial_{z}\rangle $: taking $x=y=0$  the transversal $S_0$ is a neighborhood of the origin in $M$. The constant path at the origin admits many extensions: $\Gamma \colon S_0 \times [0,1]  \to M, (\tilde{x},t) \mapsto \tilde{x}$ (the 
flow of the zero vector field),  
$\Gamma' \colon S_0 \times [0,1]  \to M, (\tilde{x},t) \mapsto e^t\tilde{x}$
 (the flow of  $z\partial_{z}$). They clearly give quite different germs of diffeomorphisms at the origin. 

\item Let $M=\R^2$ with coordinates $z,w$, and let $ {\widehat{\cF}}=\langle z\partial_{w}-w\partial_{z} \rangle$. 
Taking $x=y=0$  the transversal $S_0$ is a neighborhood of the origin in $M$. The constant path at the origin admits many extensions: $\Gamma \colon S_0 \times [0,1]  \to M, (\tilde{x},t) \mapsto \tilde{x}$ (the 
flow of the zero vector field),  or 
$\Gamma' \colon S_0 \times [0,1]  \to M, (\tilde{x},t) \mapsto R^t\tilde{x}$
where $R^t$ is rotation about the origin by the angle $t$
 (the flow of  $z\partial_{w}-w\partial_{z}$). They clearly give quite different germs of diffeomorphisms at the origin, namely, the identity and the rotation by one gradient. 
\end{enumerate}
\end{eps}
 
Now we show that different choices of $\Gamma$  induce  diffeomorphisms $S_x\to S_y$ which  differ by the flow of a vector field on $\cF$ vanishing at $x$.  
In order to do this we introduce the following notation (see \S \ref{sec:fol} for the definitions of $I_x$ and $\cF(x)$):
\begin{itemize}
\item $Aut_{\cF}(M)$ is the subgroup of local diffeomorphisms of $M$ preserving $\cF$.
\item $exp(I_x \cF)$
is the space of   time-one flows of time-dependent vector fields in $I_x \cF$.  Analogously we define
$exp(\cF(x))$. Both are subgroups of $Aut_{\cF}(M)$. 
\item In particular, applying the above to a slice $S_x$ with the restricted foliation $\cF_{S_x}$, we have: $exp(I_x \cF_{S_x})$
is the space of  time-one flows of time-dependent vector fields in $I_x \cF_{S_x}$, and    $exp(\cF_{S_x})$ consists of
 time-one flows of time-dependent vector fields in $\cF_{S_x}=\cF_{S_x}(x)$. 
Often, abusing notation, we will use the same symbols to denote germs of diffeomorphisms.
\item $GermAut_{\cF}(S_x,S_y)$ is the space of germs at $x$ of 
diffeomorphisms from the foliated manifold $(S_x,\cF_{S_x})$ to the foliated manifold  $(S_y,\cF_{S_y})$. Equivalently, it consists of germs of local diffeomorphisms in $Aut_{\cF}(M)$ mapping $S_x$ to $S_y$, restricted to $S_x$.
\end{itemize}


\begin{prop}\label{holwelldef} Let $\gamma$ be a path as above.
The class of $\Gamma(\cdot,1)  \colon S_x    \to S_y$ in the quotient $GermAut_{\cF}(S_x,S_y)/ exp(\cF_{S_x})$ is independent of the chosen extension $\Gamma$.
\end{prop}
The above quotient is given by the equivalence relation $$\psi\sim \hat{\psi}\Leftrightarrow \hat{\psi}^{-1} \circ \psi\in  
exp(\cF_{S_x})$$ 
\begin{proof}
Denote by $Z_t$ and $Z'_t$ the time-dependent vector fields in $\cF$ used to define the extensions $\Gamma$ and $\Gamma'$, and by $\phi_t$,$\phi'_t$ their flows.

For any two time-dependent vector fields $V=\{V_t\}_{t\in \R}$,$W=\{W_t\}_{t\in \R}$, the following relation between flows $\Phi_t$ at time $t$ holds \cite[eq. (2)]{Posi1988}:
\begin{equation}\label{Posieq}
\Phi_t(V+W)=\Phi_t(V)\circ \Phi_t[\{(\Phi_s(V))^{-1}_*W_s\}_{s\in \R}]
\end{equation} 
where the argument in the square bracket is the time-dependent vector field which, at time $s$, is obtained pushing forward $W_s$ via $(\Phi_s(V))^{-1}$. With $V_t:=Z_t, W_t:=Z'_t-Z_t$ the above formula reads
\begin{equation}\label{sortaBCH}
\phi'_t=\phi_t\circ (\text{time-}t \text{ flow of }\{(\phi_s)^{-1}_*(Z'_s-Z_s)\}_{s\in \R}).
\end{equation}
Since $\gamma$ is an integral curve of both $Z_s$ and $Z'_s$, we have $Z_s(\phi_s(x))=Z'_s(\phi_s(x))$ for all $s$.  Hence
$(\phi_s)^{-1}_*(Z'_s-Z_s)$ vanishes at $x$, for all $s$.  Therefore the time-$t$ flow of $\{(\phi_s)^{-1}_*(Z'_s-Z_s)\}_{s\in \R}$ lies in  $exp( \cF(x))$. This time-dependent vector field is not tangent to $S_x$, but we can obtain a vector field on $S_x$ and with the same time-1 flow on $S_x$ applying a trivial variation of Lemma \ref{vfinS} (simply replace $I_x\cF$ by $\cF(x)$ there).
\end{proof}

The notion of holonomy obtained in our attempt -- namely, the class associated to the path $\gamma$ by Proposition \ref{holwelldef} --  is unsatisfactory. Indeed, as it can be seen in examples \ref{holFx}, the ambiguity given by  $exp( \cF(x))$  is much too large to allow for linearization (i.e.,  we do not obtain a well-defined map $T_xS_x \to T_yS_y$ associated to the path $\gamma$). 
Hence we replace $\cF(x)$ by $I_x \cF$ and 
propose the following notion:
\begin{definition}\label{holfrom}
Let $(M,\cF)$ be a singular foliation, and $x,y\in M$ lying in the same leaf. 
Fix a transversal $S_{x}$ at $x$, as well as a transversal $S_y$ at $y$. A \emph{holonomy transformation  from $x$ to $y$} is an element of $$ 
\frac{GermAut_{\cF}(S_x, S_y)}{exp(I_x \cF_{S_x}) }.
$$ 
\end{definition}

\begin{remark}
The equivalence relation on $GermAut_{\cF}(S_x, S_y)$ in the above definition is finer than the one in Proposition \ref{holwelldef}, as $exp(I_x \cF_{S_x})$ is quite smaller than $exp(\cF_{S_x})$. As a consequence, holonomy transformations have a well-defined linearization (see \S \ref{sec:lintrafo}).
\end{remark}
  \begin{lemma}\label{lem:regcase}
If $x$ belongs to a regular leaf then $exp(I_x \cF_{S_x})$ is trivial.
\end{lemma}
\begin{proof}
Restricting  the slice $S_x$ if necessary, $S_x$ 
intersects transversely the leaves of $\cF$, so $\cF_{S_x}$ is the trivial singular foliation.
\end{proof}

Consider now the case of regular foliations. Lemma \ref{lem:regcase} says that the holonomy of a path is a holonomy transformation, since $exp(I_x \cF_{S_x})$  is trivial. Further, since an element of the holonomy groupoid $H$ is exactly a class of paths in $(M,\cF)$ having the same holonomy (see \cite[Cor. 3.10]{AndrSk}),   there is a canonical injective map $$H \to \{\text{holonomy transformations}\}.$$
This is the point of view that we carry over to the singular case:
 while holonomy transformations  are certainly not associated to (classes of) paths, in \S \ref{subsec:holtr} we will see that 
 they are associated   to elements of the holonomy groupoid $H$.

\subsection{The ``action'' by Holonomy Transformations}\label{subsec:holtr}

The following theorem assigns a holonomy transformation to each element of  the holonomy groupoid. It is the main result of the whole of  \S \ref{section:geomhol}. Its proof is rather involved and   will be given in Appendix \ref{holproofs}.   

\begin{thm}\label{globalaction} Let $x, y \in (M,\cF)$ be points in the same leaf $L$, and fix  transversals $S_x$ at $x$ and $S_y$ at $y$. 
Then there is a well defined map 
\begin{align}\label{Phixy}
\Phi_x^y \colon H_x^y \rightarrow \frac{GermAut_{\cF}(S_x, S_y)}{exp(I_x \cF_{S_x}) },\quad h \mapsto \langle \tau\rangle.   
\end{align}

Here $\tau$ is defined as follows, given $h \in H_x^y$:  \begin{itemize}
\item take any bi-submersion $(U,\bt,\bs)$ in the path-holonomy atlas with a point $u\in U$ satisfying $[u]=h$,
\item  take any section $\bar{b} \colon S_x \to U$ through $u$ of $\bs$ such that $(\bt\circ \bar{b})(S_x)\subset S_y$,
\end{itemize}
and define $\tau:=\bt\circ \bar{b} \colon S_x \to S_y$.
\end{thm}

\begin{cor}\label{groidmap}
For every point $x\in (M,\cF)$ fix a slice $S_x$ transverse to the foliation.
The maps $\Phi_x^y$ of Thm. \ref{globalaction}  
assemble to a groupoid morphism 
\begin{center}
\fbox{\begin{Beqnarray*}
\Phi \colon H\to \cup_{x,y}\frac{GermAut_{\cF}(S_x, S_y)}{exp(I_x \cF_{S_x}) }
 \end{Beqnarray*}}
\end{center}
where the union is taken over all pairs of points lying in the same leaf.
 \end{cor}
\begin{remark}
Heuristically, one can think of $\Phi$ as an action of the groupoid $H$ on the union of all slices.
\end{remark}
\begin{proof}
Clearly $\cup_{x,y}{GermAut_{\cF}(S_x, S_y)}$ is a set-theoretic groupoid over $M$. For every $\psi\in {exp(I_x \cF_{S_x}) }$ and $\phi \in 
{GermAut_{\cF}(S_x, S_y)}$ we have $\phi \psi \phi^{-1}\in {exp(I_y \cF_{S_y})}$. This follows from the fact that for any time-dependent vector field $Y$ on $S_x$ one has 
$\phi\circ exp(Y)\circ \phi^{-1}=exp((\phi_*Y))$, and that if $Y\in I_x\cF_{S_x}$ then $\phi_*Y\in I_y\cF_{S_y}$. Hence the target of $\Phi$ is a set-theoretic  groupoid over $M$.

We prove that $\Phi$ is a groupoid morphism,  as follows. If $u\in U$ and $v\in V$ are points of bi-submersions with $\bs_{U}(u)=\bt_V(v)$, and $\bar{b}_U,\bar{b}_V$ are as in Thm. \ref{globalaction}, then $\bar{b}_U\circ \bar{b}_V \colon S_{\bs_V(v)}\to U\circ V$ is a section through $(u,v)$ and it is given by $z\mapsto (\bar{b}_U(\bt_V(\bar{b}_V(z))),\bar{b}_V(z))$. So applying $\bt_{U\circ V}$ we obtain $(\bt_U\circ \bar{b}_U)\circ(\bt_V\circ \bar{b}_V)$.
\end{proof}

\begin{remark}\label{rem:holtransf}
\begin{enumerate}
\item If $x$ belongs to a regular leaf then Thm. \ref{globalaction} recovers the usual notion of holonomy for regular foliations, by the comments at the very end of \S \ref{subsec:sing}. For singular foliations, a point of the holonomy groupoid 
$H$ does not determine a path in $M$, nor a homotopy class of paths. 

\item The above result was inspired by \cite[Prop. 3.1]{Fernandes}, which we can not apply directly  since in general there is no Lie algebroid (defined over $M$) integrating to the holonomy groupoid $H$. It differs from \cite[Prop. 3.1]{Fernandes} in that we do not make any choices (the price to pay for this is that we have to quotient by $exp(I_x \cF_{S_x})$ ) and that the domain of the above map $\Phi$ is not a set of Lie algebroid paths but rather $H$.

\item\label{Wx}  The use of slices in Theorem \ref{globalaction} is unavoidable, as in general there exists no continuous groupoid morphism
 $$ H \to \cup_{x,y}\frac{GermAut_{\cF}(W_x;W_y)}{exp(I_x \cF_{W_x}) },$$ where $W_x, W_y$ are suitable open neighborhoods of $x, y$ respectively,
and 
where the union is taken over all pairs of points lying in the same leaf.  
This   is already apparent in the case of regular foliations. Take for example the M\"obius strip $M$ with the one-leaf foliation, so $H=M\times M$, and assume that such a  continuous groupoid morphism existed.  On one hand, the morphism property implies that the point $(x,x)\in H$ is mapped to the class of $Id_{W_x}$. On the other hand, considering $(x,y)$ as $y$ varies along a non-contractible  loop in $M$ starting and ending at $x$, by continuity  the point $(x,x)$  would be mapped to an orientation-reversing diffeomorphism of $W_x$, which can not lie in the class of $Id_{W_x}$ since elements of $exp(I_x \cF_{W_x})$  are orientation-preserving. Hence we obtain a contradiction.
\end{enumerate}
\end{remark}

The following examples for Theorem \ref{globalaction} show the dependence from the choice of module $\cF$.  

\begin{eps}\label{exglo}
\begin{enumerate}[(i)]
\item Let $M=\R$ and $ {\widehat{\cF}}$ be generated by $X:=z {\partial_z}$. We consider $0\in M$; a transversal $S_0$ is just a neighborhood of $0$ in $M$. Consider the path-holonomy bi-submersions  $U\subset M\times \R \rightrightarrows M$ defined by the generator $X$ and fix $u:=(\lambda,0)\in U$ where $\lambda$ is a real number. In order to compute the image of $[u]$ under $\Phi_0^0$ we can take any bisection passing through it. The  bisection $b \colon S_0\to \R$ with constant value $\lambda$ carries the diffeomorphism $z \mapsto exp_z(\lambda X)=e^{\lambda}z$. Hence $\Phi_0^0(
[(\lambda,0)]$ is the class of $y\mapsto e^{\lambda}y$ in $GermAut_{\cF}(S_0, S_0)/exp(I_0 X)$.  
\item Let again $M=\R$ and but now let $ {\widehat{\cF}}$ be generated by $X:=z^2 {\partial_z}$. A slice $S_0$ is just a neighborhood of $0$ in $M$. 

Consider the bi-submersions  $U\subset M\times \R \rightrightarrows M$ defined by the generator $X$ and fix $u:=(\lambda,0)\in U$ where $\lambda$ is a real number. The  bisection $b \colon S_0\to \R$ with constant value $\lambda$  carries the diffeomorphism $$z \mapsto exp_z(\lambda X)= 
\frac{z}{1-\lambda z}=z+\lambda z^2+\lambda^2 z^3+\cdots.$$
Hence $\Phi_0^0[(\lambda,0)]$ is the class of this diffeomorphism in $GermAut_{\cF}(S_0, S_0)/exp(I_0 X)$.

\item
We consider  the $S^1$-action on $M=\R^2$ by rotations. Let $U\subset M\times \R$ be the bi-submersion generated by  $y\partial_{x}-x\partial_{y}$.
Any bisection  through the origin   is given by a map  $\lambda \colon \R^2 \rightarrow \R$ (a  section of the source map $\bs$).
It induces the diffeomorphism of $\R^2$   given by 
$$ \left(\begin{array}{cc}  x  \\y \end{array}\right)
 \mapsto \left(\begin{array}{cc} cos(\lambda(x,y)) & -sin(\lambda(x,y))\\
 sin(\lambda(x,y)) & cos(\lambda(x,y)) 
   \end{array}\right)
\left(\begin{array}{cc}  x  \\y \end{array}\right),
 $$
which represents the  class $\Phi_0^0[(\lambda_0,0)]$ for $\lambda_0:=\lambda(0,0)$. 
  \end{enumerate}
\end{eps} 

 Above we fixed a choice of transversals $S_x$ and $S_y$. This choice is immaterial due to the next lemma, whose proof uses technicalities that we give in Appendix \ref{subsec:chan}.

\begin{lemma}\label{immaterial}
 Let $x,y$ be  points in a foliated manifold $(M,\cF)$ lying in the same leaf. Choose two transversals $S^i_x,$ at $x$ and 
two transversals $S^i_y$ at $y$ ($i=1,2$).
 Then the maps $$^i\Phi_x^y \colon H_x^y \rightarrow GermAut_{\cF}(S_x^i,S_y^i)/exp(I_x \cF_{S^i_x}) $$ defined in Theorem \ref{globalaction} coincide upon the canonical identification  given by eq. 
\eqref{identtransversals}.
\end{lemma}
\begin{proof}
Let $h\in H_x^y$, and take any bi-submersion $(U,\bt,\bs)$ with a point $u\in U$ satisfying $[u]=h$. For $i=1,2$   take any section $\bar{b}^i \colon S^i_x \to U$ through $u$ of $\bs$ such that $(\bt\circ \bar{b}^i)(S^i_x)\subset S^i_y$, and
 extend it to a bisection $b^i$ of $U$ such that 
$(d_xb^i)(T_xL)=0$, where $L$ denotes the leaf through $x$.
Define the local diffeomorphisms $\phi^i:=\bt\circ {b^i}$ from a neighborhood of $x$ to a neighborhood of $y$.
By Lemma \ref{s1s2} there exists $\psi^{21}_y \in exp(I_y\cF)$ mapping $S^1_y $ to $S^2_y$. Define $\psi^{12}_x$ by  requiring 
\begin{equation}\label{psieq}
\phi^2=\psi^{21}_y  \circ \phi^1 \circ \psi^{12}_x.
\end{equation}
 This equation can be reformulated as 
\begin{equation}
(\phi^2)^{-1}\circ \left(\phi^1 \circ (\phi^1)^{-1}\right )\circ \psi^{21}_y  \circ \phi^1=(\psi^{12}_x)^{-1}.
\end{equation}
Now $(\phi^2)^{-1}\circ  \phi^1 \in \exp(I_x \cF)$ by the claim in the proof of Theorem \ref{globalaction} (\cf Appendix \ref{holproofs}). Further $(\phi^1)^{-1} \circ \psi^{21}_y  \circ \phi^1$ also lies in  $\exp(I_x \cF)$. Indeed, it is   the time-1 flow of the pushforward by $(\phi^1)^{-1}$ of the time-dependent vector field in $I_y \cF$  defining $\psi^{21}_y$, and  the pushforward by $(\phi^1)^{-1}$ maps $I_y \cF$ to $I_x \cF$. Hence we conclude that $(\psi^{12}_x)^{-1}$, and therefore $\psi^{12}_x$, lies in  $\exp(I_x \cF)$.

Hence we found $\psi^{12}_x$ mapping   $S_x^2$ to $S_x^1$ lying in $exp(I_x \cF)$, and 
$\psi^{21}_y $ mapping 
 $S_y^1$ to $S_y^2$ lying in $exp(I_y \cF)$, satisfying  eq. \eqref{psieq}. 
 From Lemma \ref{ideslices} it is clear that, under the identification given by eq. \eqref
{identtransversals}, the class  $[\phi^1|_{S_x}]$ is identified with the class  $[\phi^2|_{S_x}]$.
\end{proof}

\subsection{Injectivity}\label{sec:inj} 

Let $(M,\cF)$ be a singular foliation.  We show that the morphism of groupoids $\Phi$ defined in Cor. \ref{groidmap} is injective. This allows to view $H$ as a subset of the set of holonomy transformations, providing a geometric interpretation for the elements of $H$.
Being $\Phi$ a morphism a groupoids covering $Id_M$, it suffices to show for every $x\in M$ that the map 
$$\Phi_x^x \colon  H_x^x \rightarrow \frac{GermAut_{\cF}(S_x, S_x)}{exp(I_x \cF_{S_x}) }$$
defined in eq. \eqref{Phixy} is injective.

For the sake of exposition, we first consider two extreme special cases:   a point where the foliation vanishes (Prop. \ref{prop:xinj}) and
 the case of regular foliations (Prop. \ref{prop:reginj}).
Then we prove in full generality that $\Phi$ is injective (Thm. \ref{thm:inj}).    

In all cases we will make use of:

\begin{lemma}\label{lem:notes}
Let $x \in (M,\cF)$ and $\psi$ a local diffeomorphism in $exp(I_x\cF)$. Let $(U,\bt,\bs)$ be a path holonomy bi-submersion at $x$ (whence $U \subseteq M \times \R^n$). Then the element $(x,0) \in U$ carries the diffeomorphism $\psi$.
\end{lemma}
\begin{proof}
Let $X_1,\ldots,X_n \in \cF$ such that $[X_1],\ldots,[X_n]$ are a basis of $\cF_x$. Let $W$ be a neighborhood of $x$ in $M$ such that $\cF_W = <X_1,\ldots,X_n>$. Then there exist a family of functions $f^t_1,\ldots,f^t_n \in I_x$, depending smoothly on $t\in [0,1]$, such that $\psi=  exp(X)$ is the time-one flow of the time-dependent vector field $X = \{\sum f^t_i X_i\}_{t\in [0,1]}$. 

Let $(U,\bt,\bs)$ be the path holonomy bi-submersion defined by $X_1,\ldots,X_n$ at $x$, and choose $Y_1,\ldots,Y_n \in C^{\infty}_c (U;\ker d\bs)$ (linearly independent) such that $d\bt(Y_i) = X_i$. Consider the time-dependent vector field $Y = \{\sum (f_i^t \circ \bt) Y_i\}_{t \in [0,1]}$. Define a section of $\bs$ by $b^{\psi}\colon W \to U$, where  $$b^{\psi}(y) = exp_{(y,0)}(Y).$$ Then $\bt\circ b^{\psi}=exp(X)=\psi$, so  $b^{\psi}$ is a bisection carrying $\psi$. This bisection passes through $(x,0)$ as all $f_i^t$ lie in $I_x$.
\end{proof}

\subsubsection{The case of points where the foliation vanishes}

\begin{prop}\label{prop:xinj}
Let $x\in M$ a point where the foliation $\cF$ vanishes. Then
the map $\Phi_x^x$
is injective.
 \end{prop}
\begin{proof}
We have to show for all $h\in H_x^x$: if $\Phi(h)$ is the class of $Id_{S_x}$, then $h=1_x$. Let $U$ a bi-submersion in the path-holonomy atlas and $u\in U$ with $[u]=h$. By the definition of $\Phi$, the germ $\Phi(h)$ lies in the class of $Id_{S_x}$ if{f} there is a section $b \colon S_x \to U$ of $\bs$ through $u$ such that $\psi:=\bt\circ b\in exp(I_x \cF_{S_x})$.  Since the slice $S_x$ is  an open neighborhood of $x$ in $M$, this means that $u$ carries a diffeomorphism  $\psi\in exp(I_x \cF)$.   
By Lemma \ref{lem:notes}, for any path-holonomy bi-submersion $V$ defined near $x$, the element $(x,0)\in V$ also carries $\psi$.
Hence by \cite[Cor. 2.11]{AndrSk} there exists a morphism of bi-submersions $U\to V$ with $u \mapsto (x,0)$, and therefore $h=[u]=[(x,0)]=1_x$.
\end{proof}

\subsubsection{The regular case}

The proof of the injectivity of $\Phi$ we provide here 
does not generalize to the singular case, however it has the 
virtue of being quite geometric.

\begin{lemma}\label{lem:bc} Let $(M,\cF)$ be a singular foliation and $x\in M$.
Let $U\subset M\times \R^n $ be a path-holonomy  bi-submersion and $u\in U_x$.

1) Let $b$ be a bisection  through  $u$, and  $c$ 
a bisection through $(x,0)$, and denote the diffeomorphisms they carry by $\phi_b$ resp. $\phi_c$. Then there exists a bisection of $U$ through $u$ carrying $\phi_b\circ \phi_c$.

2) If $u\in U_x^x$, then there exists a bisection of $U$ through $u$ carrying an orientation-preserving local diffeomorphism of a neighborhood $W\subset M$ of $x$.
\end{lemma}
\begin{proof}
1) There exists a morphism of bi-submersions $\alpha \colon U\circ U\to U$ with $(u,(x,0))\to u$,
since $(x,0)$ carries the identity and by
\cite[Cor. 2.11]{AndrSk}. The image of the bisection $(b,c)$ under $\alpha$ is the a bisection with the required properties.

2) Apply Prop. \ref{thm:splitting}, so $W \cong (L\cap W) \times S_x$, where $S_x$ is a slice at $x$. Denote by $s_1,\dots,s_{dim(M)}$ a set of adapted coordinates on $W$, so that there exists a set of generators $\{X_i\}_{i\le n}$ of $\cF|_W$ with $X_1=\partial_{s_1},\dots,X_k=\partial_{s_k}$. 
The corresponding path-holonomy bi-submersion is isomorphic to $U$ by 
\cite[Lemma 2.6]{AnZa11}, so we may assume that it is $U$.
The bisection $$c \colon W\cong   (L\cap W)\times S_x
\to U, s \mapsto (-2s_1,0,\dots,0)$$ carries the diffeomorphism $s\mapsto exp_{s}(-2s_1\partial_{s_1})=(-s_1,s_2,\dots,s_{dim(M)})$, which clearly is orientation-reversing. 

Now let $b$ be an arbitrary bisection through $u$. If the local diffeomorphism carried by $b$ (which fixes $x$) is not orientation-preserving, by part 1) the local diffeomorphism carried by
$b\circ c$  will be.
\end{proof}

\begin{lemma}\label{lem:calculus}
Let $B$ be an open neighborhood of the origin $0$ in $\R^k$, and $\phi$ an orientation-preserving diffeomorphism of $B$ fixing $0$. Then, shrinking $B$ if necessary, we can write $\phi=exp(X)$ where $\{X_t\}_{t\in [0,1]}$ is a time-dependent vector field on $B$ vanishing at $0$.
\end{lemma}
\begin{proof} 
It suffices to show that
$\phi$ is isotopic to $Id_B$ by diffeomorphisms fixing $0$.
A concrete isotopy from $\phi$ to its derivative
$d_0\phi$ (restricted to $B$) is given by their convex linear combination.
Now $d_0\phi$ lies in $GL_+(\R^k)$,  hence it can be connected to $Id_{\R^k}$ by a path in $GL_+(\R^k)$.
\end{proof}

Until the end of this section we consider a regular foliation $(M,\cF)$.

\begin{prop}\label{prop:reginj}
Let  $(M,\cF)$ be a manifold with a regular foliation. 
Fix a point $x\in M$ and a slice $S_x$. Then
the map $\Phi_x^x \colon H_x^x \rightarrow{GermAut_{\cF}(S_x, S_x)}$   is injective.  It follows that $\Phi$ is injective.
 \end{prop}
\begin{proof}
First notice that the target of the map $\Phi_x^x$ 
 is really as above, due to Lemma \ref{lem:regcase}.
 We have
 to show that if $h\in H_x^x$ satisfies $\Phi(h)=Id_{S_x}$, then $h=1_x$. Let $U$ a bi-submersion in the path-holonomy atlas and $u\in U$ with $[u]=h$.
There is a bisection
$b\colon W \to U$   through $u$, defined on an open neighborhood $W\subset M$, which carries an  orientation-preserving diffeomorphism $\phi$ (the bi-submersions $U$ is a composition of path-holonomy bi-submersions, so just apply Lemma \ref{lem:bc} to each of them).
 Since $\phi|_{S_x}=\Phi(h)=Id_{S_x}$, we deduce that that $\phi|_{(L\cap W)\times \{p\}}$ is orientation-preserving for all $p\in S_x$. By Lemma \ref{lem:calculus}, applied smoothly on each 
(portion of) leaf $(L\cap W)\times \{p\}$, we have
$\phi=exp(X)$ where $\{X_t\}_{t\in [0,1]}$ is a time-dependent vector field on $W$, tangent to the leaves and vanishing on the slice  $S_x$. In particular,  $X_t\in I_x\cF$ for all $t$. Hence we showed that $u\in U$ carries a diffeomorphism $\phi$ that lies in $exp(I_x\cF)$.

By Lemma \ref{lem:notes}, for any path-holonomy bi-submersion $V$ defined near $x$, the element $(x,0)\in V$ also carries $\phi$.
Hence by \cite[Cor. 2.11]{AndrSk} there exists a morphism of bi-submersions $U\to V$ with $u \mapsto (x,0)$, and therefore $h=[u]=[(x,0)]=1_x$.
\end{proof}

Consider the holonomy groupoid of a regular foliation as defined by Winkelnkemper: 
$$\cH:=\{\text{paths lying in leaves}\}/\text{holonomy}$$
where ``holonomy'' is meant in the sense of \S\ref{subsec:reg}.
In \cite[Cor. 3.10]{AndrSk} it was shown that $H$ and $\cH$ are canonically isomorphic.  
We end this section recovering this bijection in an alternative way.

Make a smooth choice of slices $S_x$ at every $x\in M$. We denote by ${GermAut}(S, S)$ the (groupoid over $M$ of) germs to automorphisms between such slices.   $\cH$ can be identified with  the image of the map $hol \colon 
\{\text{paths lying in leaves}\}  \to{GermAut(S, S)}$ introduced in \S\ref{subsec:reg},  so we can regard $\cH$ as a subset of ${GermAut}(S, S)$.

\begin{cor}\label{cHH}
 $\Phi\colon H \to {GermAut(S, S)}$ 
induces a bijection
$H\cong \cH$.  \end{cor}
\begin{proof}
To show  $\cH \subset \Phi(H)$, we let $\gamma$ be a path in a leaf of $(M,\cF)$, and need to find $h\in H$ such that $\Phi(h)=hol(\gamma)$. Cover the image of $\gamma$ by  open subsets $W_1,\dots,W_N$ such that on each $W_{\alpha}$ the foliation $\cF$ is generated by vector fields $\{X^\alpha_1,\dots,X^\alpha_n\}$. Denote by $U_{\alpha}\subset M\times \R^n$ the corresponding the path-holonomy bi-submersion.
We may assume that $\gamma$ is defined on the interval $[0,N]$ and that the velocity of $\gamma$ equals $X^\alpha_1$ when $t\in [\alpha-1,\alpha]$, for all $\alpha=1,\dots,N$. Consider the bi-submersion $U:=U_N\circ\dots\circ U_1$, and its point $$u=(u_N,\dots,u_1) \text{    where }u_\alpha:=(\gamma(\alpha-1),(1,0,\dots,0))\in U_{\alpha}, \forall \alpha=1,\dots,N.$$ Then $h:=[u]\in H$ satisfies $\Phi(h)=hol(\gamma)$.

We show $\Phi(H)\subset\cH$. Given an element $h\in H$,
take a representative $u\in U$, where $U$ is a bi-submersion in the path-holonomy atlas. Therefore $U=U_N\circ\dots\circ U_1$ for path-holonomy bi-submersions $U_{\alpha}$ ($\alpha=1,\dots,N$), and $u=(u_N,\dots,u_1)$. Notice that $U_\alpha \subset M\times \R^{dim(\cF_{\bs(h)})}$, so we can rescale its elements by defining $t\cdot(z,\lambda):=(z,t\lambda)$ for all $t\in [0,1]$ and $(z,\lambda)\in U_{\alpha}$.
 We now define a path from $\bs(h)$ to $\bt(h)$   lying in a leaf of the foliation, by
 $$\gamma \colon [0,N]\to M,\;\;\; \gamma(t)=\bt_{U_{\alpha}}((t-{\alpha}-1)\cdot u_{\alpha})
\text{ when }t\in [{\alpha}-1,{\alpha}].$$ From the definition of $\Phi$ and $hol$ it follows that $hol(\gamma)=\Phi(h)$, so $\Phi(h)\in \cH$.
 
Hence $\Phi$ maps $H$  surjectively onto $\cH$. Since $\Phi$ is injective by Prop. \ref{prop:reginj}, we obtain the desired bijection.  \end{proof}

\subsubsection{The general case}

We present a proof of the injectivity of $\Phi$ for any singular foliation $\cF$. The idea is to reduce the problem to the case where the foliation vanishes at a point.

\begin{lemma}\label{lem:trivext}
Let $x\in (M,\cF)$ and $h\in H_x^x$.
Take any bi-submersion $(U,\bt_U,\bs_U)$   with a point $u\in U$ satisfying $[u]=h$, and
  take any section $\bar{b} \colon S_x \to U$ of $\bs_U$  through $u$ such that $\tau:=\bt_U\circ \bar{b}$ maps  $S_x$ into itself. In a neighborhood $W$ of $x$ choose a diffeomorphism $W\cong I^k\times S_x$ as in Prop. \ref{thm:splitting} (the splitting theorem).

There exists a bisection\footnote{$b$ does not necessarily restrict to $\bar{b}$.}  $b\colon W\to U$ through $u$ such that the diffeomorphism it carries is the trivial extension of $\tau$ to every vertical slice, that is: 
\begin{equation}\label{diffnice}
I^k\times S_x\to I^k\times S_x, \;\;\; (\vec{s},p)\mapsto  
(\vec{s},\tau(p)).
\end{equation}
\end{lemma}
\begin{proof}

Choose generators of $\cF$ in a neighborhood $W$ of $x$ as in Prop. \ref{thm:splitting} (the splitting theorem), that is: the first elements $X_1,\dots,X_{k}$ restrict to  coordinate vector fields on $L$ and commute with all $X_i$'s, and the remaining  
elements $X_{k+1},\dots,X_n$ restrict to generators of $\cF_{S_x}$. 
Denote by $V\subset \R^n\times W$ the path-holonomy bi-submersion given by these generators. There exists a morphism of bi-submersions 
$$f\colon V\circ U\circ V\to U$$ such that $((0,x),u,(0,x))\mapsto u$.
Indeed, if we choose any bisection $b'$ of $U$ through $u$ and denote by $0$ the (constant) zero bisection of $V$, then the   bisections $0\circ b'\circ 0$ of $V\circ U\circ V$ and $b'$ of $U$ carry the same diffeomorphism, so we can apply  \cite[Lemma 2.11 b]{AndrSk}. Due to this, 
it suffices to construct a bisection of $V\circ U\circ V$ through $((0,x),u,(0,x))$ carrying the diffeomorphism \eqref{diffnice}, for composing it with $f$ we will obtain the desired bisection $b$.

Consider the following section of source map of $V\circ U\circ V$:
\begin{align*}
W &\mapsto \;\;\;\;\;\;\;\;\;\;\;\;\;V\;\;\;\circ \;\;U\;\;\circ V\\
z&\mapsto (\vec{s},\bt_U(\bar{b}(p)))\circ\bar{b}(p)\circ(-\vec{s},z).
\end{align*}
Here we write $z=(\vec{s};p)$ using  $W\cong I^k\times S_x$, and  on the right hand side $\vec{s}\in I^k$ is viewed as an element of 
$I^k\times\{0\}\subset \R^n$. Notice that $\bt_V(-\vec{s},z)=(\vec{0};p)\in S_x$. It is straightforward to check that this is a well-defined bisection of $V\circ U\circ V$ carrying the diffeomorphism \eqref{diffnice} and going  through the point $((0,x),u,(0,x))$.
\end{proof}
  
 \begin{thm}\label{thm:inj}
Let  $(M,\cF)$ be a manifold with a singular foliation.  Fix a point $x\in M$ and a slice $S_x$. Then
the map $\Phi_x^x \colon H_x^x \rightarrow  \frac{GermAut_{\cF}(S_x, S_x)}{exp(I_x \cF_{S_x}) }$   is injective.

 It follows that  $\Phi$ is injective.
 \end{thm} 
 
\begin{proof}
We have
 to show that if $h\in {H}_x^x$ satisfies $\Phi(h)=[Id_{S_x}]$, then $h=1_x$. Let $U$ a bi-submersion in the path-holonomy atlas and $u\in U$ with $[u]=h$. Let $\bar{b}\colon S_x\to U$ any section of $\bs$ through $u$ such that $\tau:=\bt\circ \bar{b}$ maps $S_x$ to itself.
 Then, by assumption, $\tau\in exp(I_x\cF_{S_x})$, that is, $\tau$ is the time-one flow of a time-dependent vector field $\{Y_t\}_{t\in[0,1]}$ on $S_x$ lying in $I_x\cF_{S_x}$.
   Choose a diffeomorphism $W\cong I^k\times S_x$ as in  the splitting theorem. By Lemma \ref{lem:trivext} there exists a bisection $b\colon W\to U$ through $u$ such that the diffeomorphism $\tilde{\tau}:=\bt\circ b$   is the trivial extension of $\tau$ to every vertical slice. In particular $\tilde{\tau}$ is the time-one flow of the time-dependent vector field on $W$ obtained extending the $Y_t$
trivially to every vertical slice, and therefore  $\tilde{\tau}  
\in exp(I_x\cF)$. 

By Lemma \ref{lem:notes}, for any path-holonomy bi-submersion $V$ defined near $x$, the element $(x,0)\in V$ carries the diffeomorphism 
$\tilde{\tau}$. As this is also the diffeomorphism carried by $b$, it follows from \cite[Prop. 2.11 b]{AndrSk} that there is a morphism of bi-submersions $U\to V$ with $u\mapsto (x,0)$, therefore $[u]=[(x,0)]=1_x$.

\end{proof}

\section{Linear  Holonomy Transformations}\label{sec:lht}

In this section we consider two notions  obtained differentiating, over a leaf $L$, the  
 map $\Phi \colon H \to \{\text{holonomy transformations}\}$ defined in Thm. \ref{globalaction}. In \S \ref{sec:lintrafo}
we consider $\Phi|_{H_L}$ as  a groupoid action and linearize it, obtaining  a groupoid representation of $H_L$ on $NL$ (first order holonomy). In \S \ref{subsection:linholrep} we differentiate  the latter to a  representation of the Lie algebroid $A_L$ on $NL$.

\subsection{Lie groupoid representations by Linear  Holonomy Transformations}\label{sec:lintrafo}
 
Here we linearise the holonomy transformations arising in  Theorem \ref{globalaction}. The result, on every leaf $L$, can be phrased as a representation (linear groupoid action) of $H_L$ on the   vector bundle $NL:=\cup_{x\in L}N_xL$. We also give a linear groupoid action of $H_L$ on the Lie algebra bundle $\g_L=\cup_{x\in L}{\g_x}$ (\cf \cite[\S 1.3]{AnZa11}) which we interpret as the adjoint representation.

 \begin{prop} \label{globallinaction}
Let $x,y$ be points of $(M,\cF)$ lying in the same leaf $L$. 

1) There is a canonical  map 
\begin{align}\label{Psixy}
\Psi_x^y \colon H_x^y &\to Iso(N_xL,N_yL).
\nonumber
\end{align}
defined as follows:  
 \begin{itemize}
\item Given $h\in H_x^y$ take any bi-submersion $(U,\bt,\bs)$ in the path-holonomy atlas with a point $u\in U$ satisfying $[u]=h$,
\item  define $$\Psi_x^y(h) \colon N_xL \to N_yL, \; [v]\mapsto [\bt_*(\tilde{v})]$$
 where $\tilde{v}\in T_uU$ is any $\bs_*$-lift of $v\in T_xM$.
\end{itemize}

2) We have $\Psi_x^y(h)= d_x\tau$, 
where the diffeomorphism $\tau$ is chosen as in Theorem \ref{globalaction}, for any choice of transversals $S_x,S_y$ and
using the canonical identifications $N_xL\cong T_xS_x$, $N_yL\cong T_yS_y$.
\end{prop}
\begin{proof}
1) Fixing $u\in U$ as above, the map $N_xL \to N_yL, \; [v]\mapsto [\bt_*(\tilde{v})]$ is well-defined since $\bt(\bs^{-1}(x))=L$ implies that $\bt_*(Ker(d_u\bs))=T_yL$. The map
$\Psi_x^y(h)$ is independent of the choice of bi-submersion $U$ and of $u$ since, for any other choice $U'$ and $u'$, there exists by definition a  morphism of bi-submersions $U\to U'$ mapping $u$ to $u'$.

2) The diffeomorphism $\tau\colon S_x \to S_y$ can be described as follows: take any (local) $\bs$-section ${b}\colon M \to U$   through $u$ such that the local diffeomorphism $\bt\circ  {b}$ of $M$ maps $S_x$ into $S_y$, and restrict $\bt\circ  {b}$ to the slice $S_x$. Now choose $v\in T_xS_x$. Then $\tilde{v}:={b}_* (v) \in T_uU$ is an $\bs_*$-lift of $v$, so $\bt_*(\tilde{v})=d_x\tau(v)$.
\end{proof}
\begin{remark}
Prop. \ref{globallinaction} 2)  says that $\Psi_x^y(h)$ can be regarded as the ``derivative at $x$'' of  $\Phi_x^y \in \frac{GermAut_{\cF}(S_x, S_y)}{exp(I_x \cF_{S_x}) }$.
(Even though $\Phi_x^y(h)$ is   an equivalence class of maps, its derivative at $x$ is well-defined by Lemma \ref{derid}.)
\end{remark}
 
\begin{cor}\label{cor:psiL}
For every leaf $L$ the map 
\begin{center}
\fbox{\begin{Beqnarray*}
\Psi_L \colon H_L \to Iso(NL,NL)
 \end{Beqnarray*}}
\end{center}
obtained assembling the  maps $\Psi_x^y$ for all $x,y\in L$, is a  groupoid morphism.  When $H_L$ is smooth, $\Psi_L$ is a Lie groupoid morphism, i.e. a representation of   $H_L$ on the vector  bundle $NL$.
\end{cor}
\begin{proof}
We first show that $\Psi_L$ is a morphism of set-theoretic groupoids. One way to deduce this is to use the fact that
$\Phi$ is a groupoid morphism (Cor. \ref{groidmap}) and then Prop. \ref{globallinaction} 2).

An alternative, direct proof is as follows. Let $g,h$ be composable elements of $H_L$ and $\xi\in T_{\bs(h)}M$. Let $u$ be a point of a bi-submersion $(U,\bt^U,\bs^U)$ (of the path-holonomy atlas) representing $g$, and similarly let $v\in V$ represent $h$. Then $gh\in H$ is represented by $(u,v)\in U \circ V$. We have
$$\Psi(gh)[\xi]=[\bt^{U\circ V}_* (X)]=[\bt^U_*(Y)]$$
where $X\in T_{(u,v)}(U\circ V)$ is a $\bs^{U\circ V}_*$-lift of $\xi$, hence $X=(Y,Z)$ where $Z$ is a $\bs^{V}_*$-lift of $\xi$ and $\bs_*^UY=\bt_*^V Z$.
On the other hand, $$\Psi(g)(\Psi(h)[\xi])= [\bt^U_*(\text{a $\bs^U_*$-lift of }\bt^V_*Z)]=[\bt^U_*(Y)].$$

Now assume that $H_L$ is smooth. For the smoothness of the map $\Psi_L$ we argue as follows.  For every bi-submersion $U$, the following locally defined map is smooth: $U_L \times NL \to NL, (u,[v])\mapsto [\bt^U_*(\tilde{v})]$,
where $\tilde{v}\in T_uU$ is any $\bs_*$-lift of $v\in T_{\bs(u)}M$. We conclude recalling that  the differentiable structure on $H_L$ is induced by the bi-submersions, see \cite[\S 2.4]{AnZa11}.
\end{proof}

 \begin{remark}\label{Psiacts}
Notice \cite[\S4]{We}\cite[\S 1.2]{CrStr} that given any Lie groupoid $\Gamma$ over a manifold $M_{\Gamma}$ and a leaf   $L_{\Gamma}\subset M_{\Gamma}$, the restriction of $\Gamma$ to the leaf acts\footnote{See eq. \eqref{eq:glln}. The map in Prop. \ref{globallinaction} 1) is obtained adapting the definition of this action to the setting of bi-submersions.} linearly on $NL_{\Gamma}$. However this can not used to  obtain the action of $H_L$ on $NL$ described in Cor. \ref{cor:psiL}, since $H$ is usually not a Lie groupoid. 
\end{remark}

\begin{eps}\label{exlin}
We calculate the linear holonomy at the origin for Examples \ref{exglo},
by taking the derivative at the origin of the holonomy transformations obtained there.
\begin{enumerate}[(i)]
\item Let $M=\R$ and $\cF$ be generated by $X:=z {\partial_z}$. Then $\Psi_0^0[(\lambda,0)] \in  Iso(T_0M, T_0M)$ is $e^{\lambda}Id_{T_0M}$.

\item Let   $M=\R$ and $\cF$ be generated by $X:=z^2 {\partial_z}$. Then $\Psi_0^0[(\lambda,0)] \in  Iso(T_0M, T_0M)$ is given by $Id_{T_0M}$.

\item
We consider  the $S^1$-action on $M=\R^2$ by rotation.  Then $\Psi_0^0[(\lambda_0,0)]\in Iso(\R^2,\R^2)$ is given by  $\left(\begin{smallmatrix}  cos(\lambda_0) & -sin(\lambda_0)\\
 sin(\lambda_0) & cos(\lambda_0) 
   \end{smallmatrix}\right).$     
\end{enumerate}
\end{eps}

\subsubsection{Adjoint representations}\label{section:adjoint}

The holonomy groupoid $H$ also acts on the isotropy Lie algebras of the foliation. Recall that by Cor. \ref{ALintegr}  the restriction $H_L$ of the holonomy groupoid $H$ to a leaf $L$ is a Lie groupoid. Here we show that the restriction to $L$ of this action is the usual adjoint representation of the Lie groupoid $H_L$ (Cor. \ref{cor:adj}).
  
The action of $H$   on the isotropy {Lie algebras} is described in the next proposition, which will be proven in appendix \ref{holproofs}. Notice that when the foliation $\cF$ is regular, we have $\g_x=\{0\}$ for all $x\in M$, so the proposition is vacuous.

\begin{prop}\label{globalactionong}
Let $x,y$ belong to the same leaf of a singular foliation. There is a canonical, well-defined map  
\begin{align}
\widehat{\Psi}_x^y \colon H_x^y \rightarrow Iso(\g_x, \g_y).
\end{align}
Under the  identification 
$\g_x\cong \cF_{S_x}/I_x\cF_{S_x}$ 
given by any slice $S_x$ at $x$ (see \cite[Rem. 1.6]{AnZa11}) and an analogue identification at $y$, the element  $h\in H_x^y$ acts by
 \begin{equation}\label{gxgymap}
[Y]\in \cF_{S_x}/I_x\cF_{S_x} \mapsto [\tau_* Y]\in \cF_{S_y}/I_y\cF_{S_y},
\end{equation}
 where the diffeomorphism $\tau \colon S_x \to S_y$ associated to $h$ is chosen as in Theorem \ref{globalaction}. 
 \end{prop}

 \begin{remark}\label{remark:actlinfol}
Let $L$ be a leaf of $(M,\cF)$ and $S_x$ a transversal to this leaf at $x$. In \S \ref{subsection:linfolnormal} we introduce the linearisation of the transversal foliation $\cF_{S_x}$; it is a foliation $ \cF_{lin}|_{N_xL}$ defined on the normal space $N_x L$.  For any $h\in H$, writing $x=\bs(h)$ and $y=\bt(h)$, the isomorphism $\Psi_{x}^{y} \colon N_xL \to N_yL$ maps the linearised foliation $ \cF_{lin}|_{N_xL}$  to $ \cF_{lin}|_{N_yL}$. This follows from the fact that the diffeomorphisms $\tau \colon S_x\to S_y$ map the foliation $\cF_{S_x}$ to $\cF_{S_y}$. As we show in lemma \ref{gxlin}, the Lie algebra $\g_x$ carries more information than the foliation $ \cF_{lin}|_{N_xL}$.
\end{remark}

For every leaf $L$, assembling the maps constructed in Prop. \ref{globalactionong} we obtain a map $\widehat{\Psi}_L\colon H_L \to Iso(\g_L)$. It is a  groupoid morphism.  (This follows from the fact that if  $\tau_i$ is a diffeomorphism associated to $h_i$ as in Theorem \ref{globalaction}, $i=1,2$, then $\tau_1\circ\tau_2$ is a diffeomorphism associated to $h_1 h_2$.)

Recall that given a transitive Lie groupoid $\Gamma$ over a manifold $L$  and  points $x,y\in L$, there is a map $\Gamma_x^y \to Hom(\Gamma_x^x,\Gamma_y^y)$ given by conjugation: an element $h\in \Gamma_x^y $ is mapped to the homomorphism $I_h : \gamma  \mapsto h \gamma h^{-1}$ (where $\gamma\in \Gamma_x^x$). Hence by differentiation we obtain a map $$\Gamma_x^y \rightarrow Iso(T_x\Gamma_x^x, T_y\Gamma_y^y), \;\;h\mapsto d_x(I_{h}).$$
We refer to the resulting representation of $\Gamma$ on $\cup_{x\in L}T_x\Gamma_x^x$ as
  the \emph{adjoint representation of $\Gamma$}. The next two statements show that the representation defined in Prop. \ref{globalactionong} is equivalent to the adjoint representation of $H_L$, provided the leaf $L$ satisfies regularity conditions.
\begin{prop}\label{prop:adjoint}
Consider the connected and simply connected Lie group $G_x$ integrating $\g_x$. Let $E \colon \g_x \to H_x^x$ the composition of the exponential map $exp : \g_x \to G_x$ with the map $\varepsilon : G_x \to H_x^x$ discussed in  \S \ref{sec:essiso}. For every $h \in H_x^y$ put $I_h(\gamma) = h\gamma h^{-1}$. Then the following diagram commutes:
\begin{align}\label{diagE}
\xymatrix{
\g_x \ar[d]_{E}  \ar[r]^{\widehat{\Psi}_x^y(h)} &  \g_y  \ar[d]^{E}  \\ 
H_x^x  \ar[r]_{I_h} &  H_y^y
}
\end{align}
\end{prop}

\begin{cor}\label{cor:adj}
Let $(M,\cF)$ be a singular foliation and $L$ a  leaf. Fix $x,y\in L$.
Then  the Lie groupoid representation 
$\widehat{\Psi}_L\colon H_L \to Iso(\g_L)$
is equivalent to the adjoint representation of $H_L$.
\end{cor}
\begin{proof}[Proof of Cor. \ref{cor:adj}]
Recall that  $H_L$ is a Lie groupoid by Thm. \ref{ALintegr}. We use the notation of Prop. \ref{prop:adjoint}, with $x,y \in L$.
By assumption, the map $\varepsilon : G_x \to H_x^x$ has discrete kernel, and further $d_0exp=Id_{\g_x}$. Hence $d_0E\colon \g_x \to T_xH_x^x$ is an isomorphism. Taking derivatives in  diagram \eqref{diagE} we obtain a commuting diagram
$$
\xymatrix{
\g_x \ar[d]_{d_0E}  \ar[r]^{\widehat{\Psi}_x^y(h)} &  \g_y  \ar[d]^{d_0E}  \\ 
T_xH_x^x  \ar[r]_{d_x(I_h)} &  T_yH_y^y.
}
$$
As this holds for all $h\in H_L$, we are done.
\end{proof}
 
Until the end of this section, we turn to the proof of Prop. \ref{prop:adjoint}.
It suffices to prove the commutativity of  diagram \eqref{diagE} locally, namely in small neighborhoods $\tilde{\g}_x, \tilde{\g}_y$ of the origin.
Indeed, any $w\in \tilde{\g}_x$ can be written as $nv$ for some natural number $n$ and $v\in \tilde{\g}_x$. We have
$I_hE(nv)=(I_hE(v))^n$ since $\varepsilon$ and $I_h$ are group homomorphisms, and we have $[E\widehat{\Psi}_x^y(h)](nv)=([E\widehat{\Psi}_x^y(h)](v))^n$ since $\widehat{\Psi}_x^y(h)$ is linear and $\varepsilon$ is a group homomorphisms.

For the convenience of the reader, before the proof we recall some facts from \S \ref{sec:essiso} (see also \cite[\S 3]{AnZa11}) that will be used in the proof: let $G_x$ be the connected and simply connected Lie group integrating $\g_x$ and $\tilde{G}_x$ be a neighborhood of the identity where the exponential map $exp \colon \tilde{\g}_x \to \tilde{G}_x$ is a  diffeomorphism.   Let $X_1,\ldots,X_n \in \cF$ induce a basis of $\cF_x$, let $(U,\bt,\bs)$ be the corresponding  path holonomy bi-submersion at the point $x$, minimal at $(x,0)$,
and let $Y_1,\ldots,Y_n \in C^{\infty}(U,\ker d\bs)$  such that $d\bt(Y_i) = X_i$ for every $1\leq i \leq n$.

Recall that $\varepsilon$ is an extension of the map $\tilde{\varepsilon} \colon \tilde{G}_x \to H_x^x$ defined as $\tilde{\varepsilon} = \sharp \circ \Delta$, where  $\Delta \colon \tilde{G}_x \to U_x^x$ is a diffeomorphism    and $\sharp \colon U_x^x \to H_x^x$ is the quotient map. Hence the restriction of $E=\varepsilon \circ exp$
 to the neighborhood $\tilde{\g}_x$  is  given explicitly by 
$$E|_{\tilde{\g}_x} \colon \tilde{\g}_x \to H_x^x,\;\;
\sum_{i=1}^l k_i [X_i] \mapsto \sharp (exp_{(x,0)}(\sum_{i=1}^l k_i Y_i)),$$
where the $\{X_i\}$ are chosen so that the first $l$ of them vanish at $x$, for $l=dim(\g_x)$.

\begin{proof}[Proof of Prop. \ref{prop:adjoint}]
Let $h \in H_x^y$ be the class of some $w$ in a bi-submersion $(W,\bt_W,\bs_W)$, so that $\bs_W(w)=x$, $\bt_W(w)=y$. Let $S_x,S_y$ be slices, $\tau \colon S_x \to S_y$ be a diffeomorphism chosen as in Theorem \ref{globalaction} using the bi-submersion $W$. Identifying $\g_p$ with $\cF_{S_p}/I_p \cF_{S_p}$ as in \cite[Rem. 1.6]{AnZa11} ($p \in \{x,y\}$), the map $\widehat{\Psi}_x^y(h) \colon \g_x \to \g_y$ is given by $[Y] \mapsto [\tau_{*}(Y)]$ (see Prop. \ref{globalactionong}). Let $\phi$ be a local diffeomorphism extending $\tau$  arising from a bisection
$b_W$ of $(W,\bt_W,\bs_W)$ through $w$. 


Choose vector fields $X_1,\ldots,X_n \in \cF$ inducing a basis of $\cF_{x}$ so that the first $l$ of them are tangent to $S_x$, and let $(U,\bt,\bs)$ be the path-holonomy bi-submersion they define. Choose lifts $Y_1,\ldots,Y_n \in C^{\infty}(U(x);\ker d\bs)$ via $\bt$ of $X_1,\ldots,X_n$ respectively. Put $X'_i = \phi_* X_i \in \cF$, $i=1,\ldots,n$. Then $X'_1,\ldots,X'_n$ induce a basis of $\cF_y$. Let $(U',\bt',\bs')$ be the path-holonomy bi-submersion they define and choose lifts $Y'_1,\ldots,Y'_n \in C^{\infty}(U;\ker d\bs')$ via $\bt'$ likewise.

Now pick an element $v = \sum_{i=1}^l k_i [X_i]$ in $\widetilde{\g}_x$. We must show that $\left(E\circ \widehat{\Psi}_x^y(h)\right) (v) = \left(I_h \circ E\right)(v)$. The left-hand side of this expression, since $\widehat{\Psi}_x^y(h)[X_i]=[X_i']$, is the class of the element $$v_{1} := exp_{(y,0)}(\sum_{i=1}^l k_i Y'_i)\in U'.$$ The right-hand side $I_h(\sharp(exp_{(x,0)}(\sum_{i=1}^l k_i Y_i)))$ is the class of $$v_2 := (w,exp_{(x,0)}(\sum_{i=1}^l k_i Y_i),w)\;\in\;{W}\circ U \circ \overline{W},$$ where $\overline{W}$ is the inverse bi-submersion of $(W,\bt_W,\bs_W)$. We will show that the elements $v_1$ and $v_2$ carry the same local diffeomorphism, whence by \cite[Cor. 2.11]{AndrSk} they quotient to the same element in $H_y^y$.

To this end, let $\mathcal{O}_x$ be an open neighborhood of $x$ in $M$ and consider the bisection\\ $b := \{exp_{(m,0)}(\sum_{i=1}^l k_i Y_i) \colon m \in \mathcal{O}_x\}$ of $(U,\bt,\bs)$ at the point $exp_{(x,0)}(\sum_{i=1}^l k_i Y_i)$. Then the bisection $b_W  \times b \times b_W$ of $W \circ U \circ \overline{W}$ contains $v_2$ and corresponds to the local diffeomorphism of $ \mathcal{O}_y := \phi(\mathcal{O}_x)$ given by $$z \mapsto  \phi(exp_{\phi^{-1}(z)}\sum_{i=1}^l k_i X_i)=  
 exp_{z}\sum_{i=1}^l k_i\phi_* X_i.$$

On the other hand, the bisection $b' := \{exp_{(z,0)}(\sum_{i=1}^l k_i Y'_i) \colon z \in \mathcal{O}_y\}$ contains $v_1$ and carries the same local diffeomorphism, since $\phi_* X_i=X_i'$. Therefore the elements $v_1$ and $v_2$ induce the same element in $H_y^y$. 
\end{proof}

\subsection{Lie algebroid representations by Linear  Holonomy Transformations}\label{subsection:linholrep}

Here we differentiate the previous holonomy transformations to representations of the transitive Lie algebroid $A_L$ over a leaf $L$. 
For the sake of simplicity of exposition, we consider only the case of \emph{embedded} leaves, even though the results hold for arbitrary immersed leaves.

We consider  the $C^{\infty}(M)$-module $\cN = \vX(M)/\widehat{\cF}$ (see Appendix \ref{subsec:normalmodule} for more details about $\cN$).   Notice that the Lie bracket of vector fields descends to a map $\widehat{\cF}\times \cN \to \cN$.


 
Now fix an embedded leaf $L$. We consider the Lie algebroid $A_L \to L$ introduced in \S \ref{sec:fol}, whose sections are given by $\widehat{\cF}/I_L\widehat{\cF}$, as well as  the normal bundle $NL= \frac{T_L M}{TL}\to L$, whose space of sections is $\cN/I_L\cN$ (see Lemma \ref{lem:NL}).
The Lie bracket induces the \emph{Bott connection on $NL$} 
$$\nabla^{L,\perp} \colon C^{\infty}(L;A_L) \times C^{\infty}(L;NL) \to C^{\infty}(L;NL),\;\; \quad \nabla^{L,\perp}_{\langle X \rangle}\langle Y \rangle = \langle [X,Y] \rangle.
$$  
The map $\nabla^{L,\perp}$  is a Lie algebroid representation of $A_L$ on $NL$, that is, a flat Lie algebroid connection, and hence can equivalently be regarded as a Lie algebroid morphism 
\begin{center}
\fbox{\begin{Beqnarray*}
\nabla^{L,\perp} \colon A_L \to Der(NL)
 \end{Beqnarray*}}
\end{center}
Here $Der(NL)$ denotes the Lie algebroid over $L$ whose sections are given by $CDO(NL)$, the first order differential operators  $D : C^{\infty}(L;NL) \to C^{\infty}(L;NL)$ such that there exists a vector field $\sigma_D \in \vX(M)$ with $D(fX) = fD(X) + \sigma_D(X)(f)X$ (\cf \cite{KCHM}).

\begin{eps}
Restricting the Lie algebroid representation $\nabla^{L,\perp}$ to the isotropy Lie algebra $\g_x$ at some point $x\in L$ 
we obtain a Lie algebra representation of $\g_x$ on the vector space $N_xL$. Explicitly, $[X ] \in \g_x\subset\widehat{\cF}/I_x\widehat{\cF}$ sends $\langle Y \rangle\in N_xL=\vX(M)/(\widehat{\cF}+I_x\vX(M)) $ to $\langle [X,Y] \rangle \in N_xL$.
We present 3 examples of this Lie algebra representation. In all 3 cases
$L=\{0\}$.
\begin{enumerate}[(i)]
\item As in Ex. \ref{exlin}, consider $M=\R$ and $\cF$ generated by $x\partial_{x}$. Then $\g_0$ is the one-dimensional Lie algebra with basis $[x\partial_{x}]$ and $N_0L=T_0M$.
Since $\cL_{x\partial_{x}}\partial_{x}=-\partial_{x}$ we conclude that the element $[x\partial_{x}]\in \g_0$ acts by $-Id_{N_0L}$.
\item\label{S1} As in Ex. \ref{exlin}, we consider 
the foliation $\cF$  generated by the  vector field $x\partial_{y}-y\partial_{x}$. Then $\g_0$ is the one-dimensional  Lie algebra with basis $[x\partial_{y}-y\partial_{x}]$ and $N_0L=T_0M$ has basis $\{\partial_{x}|_0,\partial_{y}|_0\}$. The Lie algebra element $[x\partial_{y}-y\partial_{x}]\in \g_0$ acts on $N_0L$ by $\left(\begin{smallmatrix}0 & 1 \\-1 & 0\end{smallmatrix}\right)$.
\item Consider $M=\R^2$ and $\cF$ generated by the Euler vector field $x\partial_{x}+y\partial_{y}$. Then $\g_0$ is the one-dimensional Lie algebra with basis $[x\partial_{x}+y\partial_{y}]$ and $N_0L=T_0M$. The Lie algebra element  $[x\partial_{x}+y\partial_{y}]$ acts by $-Id|_{N_0L}$.
\end{enumerate}
\end{eps}

We show that, under regularity conditions on the leaf $L$, the  
 Lie groupoid representation   $\Psi_L$ of  $H_L$ on $NL$ (see Cor. \ref{cor:psiL}) differentiates to the Lie algebroid representation 
$\nabla^{L,\perp}$ of  $A_L$. Before that, we show that in the case of regular foliations this holds not just over individual leaves but over the whole of $M$.

Recall that if $B\to M$ is a vector bundle, then  $Iso(B)$\footnote{This is called \textit{frame groupoid}  in \cite[Ex. 1.1.12]{KCHM}.} 
is the transitive Lie groupoid consisting of linear isomorphisms between fibers of $B\to M$.
 
 \begin{lemma}\label{regcasemor} Let $\cF$ be a regular foliation on $M$, so $\widehat{\cF}=C^{\infty}(M;F)$ for an involutive distribution $F$. Assume that all leaves of $\cF$ be embedded. Let $N:=\cup_{{L }}NL\to M$ where $L$ ranges over all leaves, so that $\cN=C^{\infty}(M;N)$.

Then the Lie groupoid representation $\Psi \colon H \to Iso(N)$ (see Cor. \ref{cor:psiL}) differentiates to the Lie algebroid representation $F \to Der(N)$ which, at the level of sections, is given by the map $\widehat{\cF}\times \cN \to \cN$ induced by the  Lie bracket.  
\end{lemma}


\begin{proof} We will work directly with compactly supported sections of the Lie algebroids $F$ and $Der(N)$ and with $\bs$-sections of the corresponding Lie groupoids $H$ and $Iso(N)$.
 Given $X\in \cF$, we construct a path $\gamma(\epsilon)$ of sections of the source map of $H$ with  
 $\dot{\gamma}(0)=X$ (recall that $\cF$ is the space of sections of the Lie algebroid $F$ of $H$).
Specifically, for each $\epsilon$, let $\gamma(\epsilon) \colon M\to H$ be defined as follows: $(\gamma(\epsilon))(x)$ is 
the holonomy class of the curve
$[0,1]\to M,  t \mapsto \phi^{t\epsilon}_{{X}}(x)$, where $\phi^{t}_{X}$ is the time-$t$ flow of ${X}$. 
The map $\Phi$ of Thm. \ref{globalaction} is given by holonomy along paths (see Cor. \ref{cHH}),  so  $\Phi(\gamma(\epsilon))$ is the restriction of $\phi^{\epsilon}_{X}$ to suitably chosen slices.

Hence 
 $ \Psi(\gamma(\epsilon))$ is the vector bundle isomorphism of $N$ given by  $(\phi^{\epsilon}_{X})_*$ (more precisely, the    map it induces on $N=TM/F$).
Therefore $\Psi_*(X)=\frac{d}{d\epsilon}|_{\epsilon=0}(\Psi(\gamma(-\epsilon)))$
is the first order differential operator which acts on
 $\langle Z \rangle \in  C^{\infty}(M;N)$ (the class of $Z\in\vX(M)$)
 by $$\langle Z \rangle \mapsto 
\frac{d}{d\epsilon}|_{\epsilon=0} (\phi^{-\epsilon}_{X})_* \langle Z \rangle=\langle {\cal L}_{X} Z \rangle= \langle [X, Z] \rangle.$$
\end{proof}

\begin{prop}\label{transvrep}
Let $(M,\cF)$ be a singular foliation and $L$ an embedded leaf. 
Then the  Lie groupoid representation $\Psi_L \colon H_L \to Iso(NL)$ of  Cor. \ref{cor:psiL} differentiates to the Lie algebroid representation
$\nabla^{L,\perp} \colon A_L \to Der(NL)$.
\end{prop}
\begin{proof}
 Given $\tilde{X}\in \cF/I_L\cF=C^{\infty}_c(L;A_L)$ defined in a neighborhood of $x\in L$,   we construct a path $\gamma(\epsilon)$ of sections of the source map of $H_L$ with  
 $\dot{\gamma}(0)=\tilde{X}$. We may assume that $(\tilde{X} \text{ mod } I_x\cF) \in \cF_x$ is non-zero, for otherwise it acts trivially.
Choose a lift $X\in \cF$.  Starting from a basis of $\cF_x$, we can construct 
 generators $\{X_i\}_{i\le n}$ of $\cF$ in a neighborhood $M_0$, with $X_1=X$.
  Denote by $U$ the corresponding path-holonomy bi-submersion, and fix vertical lifts $\{Y_i\}_{i\le n}$ of the $X_i$ w.r.t. the target map $\bt$. 
 For every $\epsilon$ sufficiently close to zero, the map $$\hat{\gamma}(\epsilon) \colon M_0\to U, y \mapsto exp_{(y,0)}(\epsilon   Y_1) $$
is a bisection of $U$ carrying $\phi^{\epsilon}_{X}$, the time-$\epsilon$ flow of $X$.
The path $\epsilon\mapsto\gamma(\epsilon):=\sharp (\hat{\gamma}(\epsilon))|_{M_0\cap L}$ is a path of sections of the source map of $H_L$  with  
 $\dot{\gamma}(0)=\tilde{X}$.  

Using Prop. \ref{globallinaction} 2) we obtain that  $ \Psi_L(\gamma(\epsilon))$ is the vector bundle isomorphism of $NL$ given by  $(\phi^{\epsilon}_{X})_*$ (or rather the map it induces on $NL$). Therefore $(\Psi_L)_*(\tilde{X})=\frac{d}{d\epsilon}|_{\epsilon=0}(\Psi_L(\gamma(-\epsilon)))$
is the first order differential operator which acts on  $\langle Z \rangle \in C^{\infty}(L;NL)=\vX(M)/(\cF+I_L\vX(M))$ (see Lemma \ref{lem:NL}) by $$\langle Z \rangle \mapsto 
\frac{d}{d\epsilon}|_{\epsilon=0} (\phi^{-\epsilon}_{X})_* \langle Z \rangle=\langle {\cal L}_{X} Z \rangle= \langle [X, Z] \rangle=\nabla^{L,\perp}_{\tilde{X}}\langle Z \rangle.$$
\end{proof}

Finally, recall that every transitive Lie algebroid acts by the bracket on its isotropy bundle (the kernel of the anchor map), see \cite[Ex. 3.3.15]{KCHM}. In the case of $A_L$ this representation is  $$\widehat{\nabla} \colon A_L \to Der(\g_L),\quad \widehat{\nabla}_X(V) = [X,V]_{A_L},$$
where $\g_L$ denotes the isotropy bundle of $A_L$ (a bundle of Lie algebras whose fiber over $x\in L$ is canonically isomorphic to $\g_x$).
 By $Der(\g_L)$ we denote the Lie algebroid over $L$ whose sections are covariant differential operators on the vector bundle $\g_L$ which are derivations of the bracket on the fibres of $\g_L$.

The Lie groupoid representation 
$\widehat{\Psi}_L\colon H_L \to Iso(\g_L)$ (see Prop. \ref{globalactionong}) 
differentiates to the  representation $\widehat{\nabla} \colon A_L \to Der(\g_L)$. This follows immediately from  Cor. \ref{cor:adj} and \cite[Prop. 3.7.4]{KCHM}.

\section{The linearized foliation near a leaf}\label{section:linfol}

In this section, for any leaf $L$, we show that there is 
well-defined notion of ``linearization of $\cF$ at $L$.''
  In \S \ref{subsection:linfolnormal} we realize the linearized foliation as a foliation on $NL$, and show that it is induced by a Lie groupoid action, provided $L$ is embedded. 
  In \S \ref{subsection:linfolquotient} we give an alternative description of the linearized foliation as a foliation on  $({H_x\times N_xL})/{H_x^x}$, under the same assumptions. In these two subsections $L$ is taken to be embedded (and not immersed) just to simplify the exposition.
   Finally, in \S \ref 
{subs:remlin}, we make some comments on the linearization problem:
under what conditions is $\cF$ isomorphic to its linearization nearby a leaf? Notice that an answer to this question would constitute a version of the Reeb Stability Theorem for singular foliations. We also briefly discuss the relation to singular Riemannian foliations.

\subsection{The linearized foliation on the normal bundle}\label{subsection:linfolnormal}

Let $L$ be an embedded leaf of the singular foliation $(M,\cF)$.  
 There is a canonical identification $$I_L/I^2_L \cong \Gamma(N^*L)=C^{\infty}_{lin}(NL),\;\; [f] \mapsto df|_L,$$ where $I_L$  
  denotes the functions on $M$ that vanish on $L$,   $C^{\infty}_{lin}(NL)$ the fiberwise linear functions on the normal bundle, and 
  $[f]:=(f \text{ mod }I^2_L)$ for $f\in I_L$.   
   
Given a vector field $Y$ on   $M$   tangent to $L$, we 
denote by $Y_{lin}$ the vector field on $NL$ which acts as $Y|_L$ on the fiberwise constant functions, and as follows on $C^{\infty}_{lin}(NL)\cong I_L/I^2_L$: for all $f\in I_L$,  $$Y_{lin}[f]:=[Y(f)].$$

We obtain a bracket-preserving assignment $$lin \colon \cF \to \vX_{lin}(NL),\;\;Y \mapsto Y_{lin},$$
where $\vX_{lin}(NL)$ denotes the vector fields on $NL$ which preserve the fiberwise constant functions and $C^{\infty}_{lin}(NL)$. 
Notice that $lin$ factors to a map $\cF/I_L\cF \to \vX_{lin}(NL)$.
 \begin{definition}\label{def:linF}
The \emph{linearization of $\cF$} is  the foliation $\cF_{lin}$ on $NL$
generated (as a $C^{\infty}(NL)$-module) by $\{Y_{lin}:Y \in \cF\}$. 
\end{definition}

\begin{remark}\label{gxlin} The linearization procedure, obviously, is not injective. This is already apparent if one considers a fiber $N_xL$ of the normal bundle. More precisely,
by Lemma \ref{lem:com} it is clear that restricting the  map (obtained factoring) $lin$  we obtain a
canonical surjective Lie algebra morphism $\g_x \to \cF_{lin}|_{N_xL}:=\{Z|_{N_xL}: Z \in \cF_{lin} \text{ is tangent to }N_xL\}$. Explicitly, it is  given by $$\g_x\to \cF_{lin}|_{N_xL},\;\langle Y \rangle \mapsto (Y_{lin})|_{N_xL}.$$
 This map is usually not injective (take for instance $(M,\cF) = (\R,x^2\partial_{x})$ and $x=0$). 
  \end{remark}

Recall that we defined the Lie algebroid morphism
$\nabla^{L,\perp} \colon A_L \to Der(NL)$ in \S\ref{subsection:linholrep}. At the level of section it induces $\nabla^{L,\perp} \colon \cF/I_L\cF \to CDO(NL)$, which is essentially given by the Lie bracket of vector fields.

\begin{lemma}\label{lem:com}
The following diagram commutes:
$$ \xymatrix{ 
\cF \ar[d]\ar^{lin}[r]& \vX_{lin}(NL)\ar^{\cong}[d]\\
\cF/I_L\cF  \ar[r]^{\nabla^{L,\perp}}& CDO(NL).
}$$
Here the right arrow is the Lie algebra isomorphism given by \begin{equation}\label{eq:cdolin}
\vX_{lin}(NL)\cong CDO(NL), Y \mapsto (a\mapsto [Y,a]_{\vX(NL)}),
\end{equation}
where $a\in \Gamma(NL)$ is also interpreted as a vertical (constant) vector field on $NL$.
\end{lemma}
\begin{proof}
Let $Y\in \cF$. We have to check that $[Y_{lin},\cdot]_{\vX(NL)}=
[Y,\cdot]|_L \text{ mod } TL$ when acting on elements of $\Gamma(NL)$, i.e. that
 \begin{equation}\label{eq:ver}
[Y_{lin},Z_{ver}]_{\vX(NL)}=
([Y,Z])_{ver}
\end{equation}
for all $Z\in \vX(M)$, where $Z_{ver}:=
(Z|_L \text{ mod } TL)\in \Gamma(NL)$.

To this aim, take $f\in I_L$.
For the r.h.s. we have
$$([Y,Z])_{ver}[f]=[Y,Z](f)|_L=\big(Y(Z(f))-Z(Y(f))\big)|_L.$$
For the l.h.s.
\begin{align*}
[Y_{lin},Z_{ver}]_{\vX(NL)}[f]&=Y_{lin}(Z_{ver}[f])-Z_{ver}(Y_{lin}[f])\\
&=Y_{lin}(Z(f)|_L)-Z_{ver}[Y(f)]\\
&=\left(Y(Z(f))\right)|_L-\left(Z(Y(f))\right)|_L
\end{align*}
where we used $Z_{ver}[f]=(Z(f))|_L\in C^{\infty}(L)$, $Y_{lin}[f]=[Y(f)]\in C_{lin}(NL)$ and $Y_{lin}(g|_L)=Y(g)|_L\in C^{\infty}(L)$ (for $g:=Z(f)$).

Both sides of eq. \eqref{eq:ver} are vertical (constant) vector fields on $NL$, and as we just showed that their action on fiberwise linear functions agree, we conclude that they are equal.
\end{proof}

From Lemma \ref{lem:com} we obtain immediately:
\begin{cor}\label{cor:agree}
The subset $\{Y_{lin}:Y \in \cF\}$ of $\vX_{lin}(NL)$ agrees with the image of $\nabla^{L,\perp} \colon \cF/I_L\cF \to CDO(NL)$, under the  identification \eqref{eq:cdolin}.
\end{cor}

Given an action of a  Lie groupoid $G\rightrightarrows M$ on a map $\pi\colon N\to M$, the \emph{induced foliation} on $N$ is the one generated (as a $C^{\infty}(N)$-module) by the image of the corresponding infinitesimal action
$\Gamma(A)\to \vX(N)$, where $A$ is the Lie algebroid of $G$. 

\begin{cor}\label{cor:ind} Let $(M,\cF)$ be a singular foliation and $L$ an embedded leaf. 
Then the linearized foliation $\cF_{lin}$ is the foliation induced by the Lie groupoid action $\Psi_L$  of $H_L$ on $NL$ (see Cor. \ref{cor:psiL}).
\end{cor}
\begin{proof}  
By Prop. \ref{transvrep} the above Lie groupoid action differentiates to the Lie algebroid action of $A_L$ on $NL$ which, under the identification \eqref{eq:cdolin}, is given by $\nabla^{L,\perp}$. Now use Cor. \ref{cor:agree}. 
\end{proof}

\subsection{The linearized foliation on  \texorpdfstring{$({H_x\times N_xL})/{H_x^x}$}{Lg}}\label{subsection:linfolquotient}

We give an alternative description of the linearization of $\cF$ at $L$ defined in Def. \ref{def:linF}, in the case that $L$ is
an embedded leaf. 
Fix a point $x$ in $L$. Since $H_L$ is a Lie groupoid (Thm. \ref{ALintegr}), $H_x\to L$ is a principal $H_x^x$-bundle. There is a linear action of $H_x^x$ on $N_xL$ given by $\Psi$ (see Prop. \ref{globallinaction}). We consider the associated vector bundle, that is,
$$Q:=\frac{H_x\times N_xL}{H_x^x},$$
where the action of $H_x^x$ on $H_x\times N_xL$ is given by
$g\cdot(h,\xi)=(hg^{-1},\Psi(g)\xi)$.

\begin{lemma}
There is a canonical diffeomorphism
\begin{equation*}
\Upsilon\colon Q \to NL,\;\; [(h,\xi)]\mapsto \Psi(h)\xi.
\end{equation*}
\end{lemma}
\begin{proof}
$\Upsilon$ is well-defined, since for all $g\in H_x^x$ we have
$\Psi(hg^{-1})(\Psi(g)\xi)       =\Psi(h)\xi$ as a consequence of the fact that $\Psi$ is a groupoid morphism, see Cor. \ref{cor:psiL}.
$\Upsilon$ is surjective: given $\eta \in N_yL$, pick any $h\in H_x^y$ and define $\xi:=(\Psi(h))^{-1}\eta$. Then $\Upsilon[(h,\xi)]=\eta$.
Using the fact that $\Psi$ is a groupoid morphism, it is straightforward to check that $\Upsilon$ is injective.
\end{proof}

 There is a map 
$$ \Gamma_c(A_L)=\cF/I_L\cF \to \vX(Q)$$
defined as follows.
 Given $Y\in \Gamma(A_L)=\Gamma(ker (\bs_*|_M))$, consider the vector field $(R_*Y,0)$ on 
 $H_x\times N_xL$, defined to be  $((R_h)_*Y,0_\xi)$ at the point $(h,\xi)$. This vector field is $H_x^x$-invariant, and hence projects to a vector field   on $Q$, which we denote by $[(R_*Y,0)]$. 
 
\begin{prop}\label{prop:equiv2m} Let  $L$ be
an embedded leaf of $(M,\cF)$.
Under the canonical diffeomorphism $\Upsilon$:

1) the action $\Psi_L$ of $H_L$
on $NL$ corresponds to the action of $H_L$ on $Q$ induced by left groupoid multiplication on $H_x$

2) $\{Y_{lin}:Y \in \cF\}$ on $NL$ corresponds to  $\{[(R_*Y,0)]:Y\in \Gamma_c(A_L)\}$ on $Q$. 
\end{prop}
\begin{proof}
1) follows from the definition of $\Upsilon$.

2) follows from 1), from Cor. \ref{cor:ind}, and from the fact that the infinitesimal generators of left groupoid multiplication are the right-invariant vector fields.
\end{proof}
 
\begin{remark}
Consider the case when the foliation $\cF$ is regular. Then the leaves of the (regular) foliation $\cF_{lin}$ on $NL$ are given by pasting together the flat sections of the flat connection $\nabla^{L,\perp}$ (see Prop. \ref{cor:ind}).
Further, the leaves of the above (regular) foliation on $Q$ are obtained projecting the foliation on ${H_x\times N_xL}$ by horizontal copies of $H_x$.
\end{remark} 
 
\subsection{Remarks on the linearization problem}\label{subs:remlin}

Consider a singular foliation $(M,\cF)$ and an embedded leaf $L$.

\begin{definition}\label{def:lina}
We say that \emph{$\cF$ is linearizable about $L$} if
 there exist  tubular neighborhoods $U\subset M$ of $L$ and $V\subset NL$ of the zero section, and a  diffeomorphism $U\to V$ mapping $L$ to the zero section
inducing an isomorphism  (of $C^{\infty}(U)$-modules) between
 the foliation $\cF|_U$ and $\cF_{lin}|_V$, where $\cF_{lin}$ is the linearized foliation on $NL$.
\end{definition}
Notice that  it does not make sense to extend Def. \ref{def:lina} to immersed leaves $L$, as the zero section is always an embedded submanifold of $NL$. 
\begin{remark}
Let $\cF$ be generated by just one vector field $X$ vanishing at a point $x$. Then $\cF$ is  linearizable about $\{x\}$ iff there exists a smooth, nowhere vanishing function $f$ defined on $V\subset T_xM$ and  a diffeomorphism $U\to V$ that maps $X$ to $f\cdot X_{lin}$. In particular, the linearizability of the foliation $\cF$ about $\{x\}$ is  a weaker condition than the linearizability of the vector field $X$ about $x$.
\end{remark}
 
\begin{remark}
It is known that not every singular foliation comes from a Lie algebroid \cite[Prop. 1.3]{AnZa11}.
(The local version of the question is still open.) 
However if $\cF$ is linearizable about an embedded leaf $L$, then $\cF|_U$ comes from a Lie algebroid, where $U$ is some tubular neighbourhood of $L$. This follows from Cor. \ref{cor:ind}. A refinement of this statement will be given in Prop. \ref{prop:proper} (see also Remark \ref{rem:linearization} \ref{rem:a}).
\end{remark}

Recall that a 
Lie groupoid  $G\rightrightarrows U$ is \emph{proper}
if its target-source map $(\bt,\bs) \colon G\to U\times U$ is proper (i.e., preimages of compact sets are compact). Properness implies that $G$ is \emph{proper at $y$}, for all $y\in U$ (see \cite[\S 1]{CrStr}).

\begin{prop}\label{prop:proper}
Let $(M,\cF)$ be a singular foliation, and $x$ a point of an embedded leaf $L$ so that the Lie groupoid $H_L$ is Hausdorff.   The following are equivalent: 
\begin{itemize}
\item[a)] $\cF$ is linearizable about $L$ and $H_x^x$ is compact
\item[b)] there exists a tubular  neighborhood $U$ of $L$  and 
a Hausdorff Lie groupoid $G\rightrightarrows U$, proper at $x$, inducing the foliation $\cF|_U$.\end{itemize}
\end{prop}

\begin{proof}
$a)\Rightarrow b).$
By Cor. \ref{cor:ind}, the linearized foliation $\cF_{lin}$ is   induced by a Lie groupoid action of $H_L$ on $NL$. In other words, $\cF_{lin}$ is the foliation induced by the transformation groupoid $H_L \ltimes NL \rightrightarrows NL$. This transformation groupoid is proper by Lemma \ref{lem:prop} below (applied to $H_L$), so in particular it is proper at $x\in L\subset NL$. Now use the assumption that $\cF$ is linearizable.

$b)\Rightarrow a).$ The linearization theorem for proper groupoids, in the version of Crainic-Struchiner \cite[Thm. 1]{CrStr}, implies that there is an isomorphism from $G$ to $G_{lin}|_V$, where we shrink $U$ if necessary and $V$ is a neighborhood of the zero section of $NL$. 
Here  the Lie groupoid $G_{lin}\rightrightarrows NL$ is the linearization of $G$ about $L$, which induces the foliation $\cF_{lin}$ on $NL$ by Lemma \ref{lem:lingrfol} below.
In particular, the base map of the above isomorphism is a diffeomorphism $U\to V$ that sends $\cF|_U$ to $\cF_{lin}|_V$, hence $\cF$ is linearizable about $L$.
Further, 
since $G$ is proper at $x$, its isotropy group $G_x^x$ is compact. Hence $H_x^x$, which is a quotient of $G_x^x$ due to \cite[Ex. 3.4(4)]{AndrSk}, is also compact. 
 \end{proof}

\begin{remarks}\label{rem:linearization}
\begin{enumerate}
\item\label{rem:a} 
In general, a foliation linearizable about  a leaf $L$ does not come  from a  Lie groupoid
proper at $x$ (where $x\in L$).   For instance, the foliation on $\R^2$ induced by the   action of $SL(2,\R)$ on $\R^2$ (see \cite{AndrSk}, \cite{AnZa11}) is linearizable around zero (as the action is linear). However this foliation does not arise from a groupoid $G$ proper at $0$: if it did, $H_0^0$ would be  compact (being a quotient of $G_0^0$), contradicting the non-compactness of  $H_0^0  = SL(2,\R)$. The same applies, of course, if we replace $SL(2,\R)$ by $GL(2,\R)$, in fact for any foliation defined by a linear action of a non-compact group (see \cite[Ex. 4.6]{AnZa11}).

The situation is different if one assumes the set-up of Prop. \ref{prop:proper}. More precisely: assuming the leaf $L$ at $x \in M$ is embedded and  $H_x^x$ is compact, the foliation is linearizable at $L$ if{f} it comes from a Lie groupoid proper at $x$.

\item Notice that a linearization theorem for singular foliations 
would extend the Reeb stability theorem for regular foliations  (\cf \cite[Thm. 2.9]{MM}). Instances of the  linearization theorem in which additional structure is present were recently established: for the symplectic foliation of a Poisson manifold   by Crainic-Mar{c}u\c{t} in \cite{CrMar}, and for orbits of a proper Lie groupoid by Crainic-Struchiner in \cite{CrStr}. 

\item A way to approach such linearization results is by understanding certain averaging processes along a certain groupoid. The difficulty with a general singular foliation is that, although the holonomy groupoid is always longitudinally smooth (see \cite{Debord}), its $\bs$-fibers usually have jumping dimensions (unlike the case studied in \cite{CrStr}). For this reason one needs to somehow modify the usual averaging processes in order to obtain linearization conditions for an arbitrary singular foliation. This is a problem of different order, which deserves to be addressed in a separate article.

In fact, in view of the previous remarks, it seems reasonable to think that the compactness of $H_x^x$ might be one of the sufficient conditions for the singular version of the Reeb stability theorem. 
\end{enumerate}
\end{remarks}

We end   with two Lemmas {used in} the proof of Prop. \ref{prop:proper}.

\begin{lemma}\label{lem:prop}
Let $\Gamma\rightrightarrows L$ be a transitive Lie groupoid such that for some $x\in L$ the isotropy group $\Gamma_x^x$ is compact.
Then, for any action of $\Gamma$ on a surjective submersion $\pi \colon N\to L$, the corresponding transformation groupoid $\Gamma \ltimes N\rightrightarrows N$ is proper.
\end{lemma}
\begin{proof}
First we point out that $\Gamma$ is a proper groupoid.  Indeed $\Gamma$ is isomorphic to the gauge groupoid for the principal $\Gamma_x^x$-bundle   $\Gamma_x\to L$, whose fiber is compact, so that the target-source map for the gauge groupoid is proper.

Given any compact subsets $K_1,K_2\subset N$, consider the preimage  $\psi^{-1}(K_1\times K_2)$  under the target-source map of the transformation groupoid, that is, under 
$\psi\colon (\Gamma \ltimes N)\to N\times N, (g,\xi)\mapsto (g\xi, \xi)$. Clearly $$\psi^{-1}(K_1\times K_2)=\{(g,\xi)\in \Gamma \ltimes N: g\xi\in K_1,  \xi \in K_2\}$$ is contained in $(\bt^{-1}(\pi(K_1))\cap \bs^{-1}(\pi(K_2)))\times K_2$, which is compact since $\Gamma$ is a  proper groupoid. (Here $\bt,\bs$ are the target and source maps of $\Gamma$.) Hence $\psi^{-1}(K_1\times K_2)$, being a closed subset of a compact one,  is compact.
\end{proof}

Let $G\rightrightarrows M$ be a Lie groupoid.  Let  $L$ be a leaf, and denote by $G_L$ the restriction of $G$ to $L$. As we pointed out in Rem. \ref{Psiacts}, there is an action of $G_L$ on the normal bundle $NL=TM|_L/TL$.
Given a bisection $\gamma \colon L \to G_L$ (viewed as a section of the source map of $G_L$), its action on $NL$ is given by the  bundle automorphism 
\begin{equation}\label{eq:glln}
(\bt \circ b)_*|_{L},
\end{equation}
 where $b \colon M \to G$ is any bisection of $G$ extending $\gamma$.
The \emph{linearization of $G$ about $L$}, denoted by  $G_{lin}$, is the transformation groupoid $G_L\ltimes NL\rightrightarrows NL$ of this action \cite[\S 1.2]{CrStr}. 
The following lemma states that the foliation of the linearized groupoid is the linearization of the groupoid foliation.

\begin{lemma}\label{lem:lingrfol}
Let $G\rightrightarrows M$ be a Lie groupoid, and denote by $\cF$ the induced foliation. Let  $L$ be an embedded leaf. Denote by $G_{lin}$ the linearization of $G$ about $L$, and denote by $\cF(G_{lin})$ the foliation  on $NL$ it induces.
Then  $\cF_{lin}=\cF(G_{lin})$.
\end{lemma}
\begin{proof}
First observe that, for any family $\phi_{\epsilon}$ of diffeomorphisms of $M$ preserving $L$ with $\phi_0=Id_M$,   the linearization of the vector field $\frac{d}{d\epsilon}|_0\phi_{\epsilon}$ is obtained applying $\frac{d}{d\epsilon}|_0$  to  the linearization   of the $\phi_{\epsilon}$'s, i.e. to 
$(\phi_{\epsilon})_*|_L$ (seen as an automorphism of $NL$). In particular, given   a family of bisections
 $b_{\epsilon} \colon M \to G$  with $b_0=Id_M$, the linearization of the vector field $\frac{d}{d\epsilon}|_0(\bt\circ b_{\epsilon})$ is exactly $\frac{d}{d\epsilon}|_0(\bt\circ b_{\epsilon})_*|_L$.

Both the foliations $\cF_{lin}$ and $\cF(G_{lin})$ are generated by linear vector fields on $NL$ (i.e. elements of $\vX_{lin}(NL)$): $\cF_{lin}$ by its very definition, and $\cF(G_{lin})$ because $G_{lin}$ is the transformation groupoid of an action by vector bundle automorphisms of $NL$. So it suffices to consider linear vector fields.

We have  $$\widehat{\cF}=\{\frac{d}{d\epsilon}|_0(\bt\circ b_{\epsilon})\},$$ where $b$ ranges over all families  $b_{\epsilon} \colon M \to G$ of bisections of $G$ with $b_0=Id_M$.

The linear vector fields in $\widehat{\cF(G_{lin}})$ are the infinitesimal generators of the $G_L$ action on $NL$. By eq. \eqref{eq:glln} they are exactly
$$\{\frac{d}{d\epsilon}|_0(\bt\circ b_{\epsilon})_*|_L\},$$
where for any  family  $\gamma_{\epsilon} \colon L \to G_L$ of bisections of ${G_L}$ with $\gamma_0=Id_L$ we choose a family of bisections
 $b_{\epsilon} \colon M \to G$ of $G$ with $b_0=Id_M$  extending it.
 
From the above observation it is clear that the set of linear vector fields in $\widehat{\cF}_{lin}$ and $\widehat{\cF(G_{lin})}$ agree, and therefore that $ \cF_{lin}$ and $ \cF(G_{lin})$ agree.
\end{proof}

\section{Further applications}

In this short section we present two applications of the material developped so far.

\subsection{Riemannian foliations}\label{subsec:rf}

 The following definition, when specialized to the case of regular foliations, reduces to the notion of Riemannian foliation given in the Introduction.

\begin{definition}\label{def:ada}Let $(M,\cF)$ be a foliated manifold. We say that a Riemannian metric $g$ on $M$ is \emph{adapted} if, for all leaves $L$, the bundle metric on $NL$ (induced by   $NL\cong (TL)^{\perp}$) is preserved by 
the action $\Psi_L$ of $H_L$ defined in Cor. \ref{cor:psiL}. 
\end{definition}
It would be interesting to study the geometric properties of (singular) foliations with an adapted metric. Not every foliation admits an adapted Riemannian metric, but foliations arising from proper Lie groupoids do, by the work of Pflaum-Posthuma-Tang \cite{PPT}. We present a further existence result, which relies on \cite{PPT}.

\begin{prop}\label{prop:Riem}
Let $(M,\cF)$ be a singular foliation and let $L$ be an embedded leaf. Assume that $\cF$ is linearizable about $L$ and that the Lie group  $H_x^x$ is compact, where $x\in L$.
 
Then there is a tubular neighborhood $U$ of $L$ such that $(U,\cF|_U)$ admits an adapted, complete Riemannian metric which makes it into a singular Riemannian foliation.
\end{prop}

Recall that a \emph{singular Riemannian foliation} in Molino's sense (\cite[\S 6]{Mo88}, see also \cite[\S 1]{ABT}) consists of manifold $M$  with a suitable\footnote{More precisely, it is required that  every tangent vector to a leaf can be extended to a vector field   on $M$ tangent to the leaves.} partition into immersed submanifolds (leaves), together with 
a complete Riemannian metric such that if a  geodesic  intersects perpendicularly a leaf, it intersects perpendicularly every leaf it meets.

\begin{proof}
By (the proof of) Prop. \ref{prop:proper}, there exists a neighborhood U and a proper Lie groupoid $\Gamma$ over $U$ inducing the foliation $\cF|_U$. Hence $U$ admits a complete  Riemannian metric $g$ whose transversal component is preserved by the action of $\Gamma$, by   \cite[Prop. 3.13]{PPT} 
(cf. also \cite[Def. 3.11]{PPT}). The metric $g$ is adapted (in the sense of Def. \ref{def:ada}), as a consequence of the fact that $H|_U$ is a quotient of $\Gamma$ \cite[Ex. 3.4.(4)]{AndrSk}.
Further,  $g$  makes the leaves of the foliation $\cF|_U$ into singular Riemannian foliation by \cite[Prop. 6.4]{PPT}. 
\end{proof}

\subsection{Deformations of singular foliations}\label{section:deform}

As a further application of the Bott connection $\nabla^{L,\perp}$ introduced in \S \ref{subsection:linholrep}, we look at deformations of singular foliations. 
We consider deformations of singular foliations (keeping the underlying $C^{\infty}(M)$-module structure fixed),
and we attach to them a class in foliated cohomology. 

\subsubsection{Deformations}\label{subsection:deform}


A regular foliation $F\subset TM$ on a manifold $M$ can be equivalently described as a Lie algebroid over $M$ with injective anchor\footnote{More precisely: if $F\subset TM$ is a regular foliation, then the restriction of the Lie bracket of vector fields $[\cdot,\cdot]_{\vX(M)} $ and the inclusion into $TM$ make $F$ into a Lie algebroid with injective anchor. If $(A,[\cdot,\cdot]_A,\rho_A)$ is a Lie algebroid and $\rho_A$ is injective, then it is canonically isomorphic to $(\rho_A(A),[\cdot,\cdot]_{\vX(M)}, \text{ inclusion})$
and $\rho_A(A)$ is a regular foliation.}. Since injectivity is an open condition, a smooth family of foliations of $M$ parametrized by $t\in(-\epsilon,\epsilon)$ which agrees with $F$ at $t=0$ is the same thing as a smooth family of Lie algebroids over $M$ which agrees with $F$ at $t=0$. We now consider the case of singular foliations.

Recall  that a \emph{Lie-Rinehart algebra} (an algebraic version of the notion of Lie algebroid) over  a unital commutative algebra $\cC$ consists of a unital left $\cC$-module $\cM$, a Lie algebra structure $[\cdot,\cdot]$ on $\cM$, and 
a map $\rho \colon \cM \to Der(\cC)$ which is a $\cC$-module morphism, a Lie algebra morphism, and satisfies  the Leibniz rule (\cf \cite[\S 1]{Huebschmann}). 
An isomorphism between Lie-Rinehart algebras $(\cC,\cM)$ and $(\cC',\cM')$ consists of an algebra isomorphism $\cC\to \cC'$ and a Lie algebra isomorphism $\cM\to \cM'$ which intertwine both the module structures and  $\rho,\rho'$.
 
In the following, for a given vector bundle $E\to M$, we take $\cC=
C^{\infty}(M)$ and $\cM$ a locally finitely generated $C^{\infty}(M)$-submodule of $\Gamma(E)$, so that $\cM$ is a $\cC$-module in a natural way.
\begin{definition}\label{dfn:smvar}
We say that 1-parameter family $\{m_t\}_{t \in \R}$ of elements of $\cM$ \emph{varies smoothly with $t$} if it does when regarded as a 1-parameter family of elements of $\Gamma(E)$, i.e. if $\{(t,m_t):t\in \R\}$ is a smooth section of the vector bundle over $\R\times M$ obtained pulling back $E$. We say that 1-parameter family of linear maps $\phi_t \colon \cM\to \cM$ varies smoothly with $t$ if, for every $m\in \cM$, the 1-parameter family $\{\phi_t(m)\}$ of elements of $\cM$ varies smoothly with $t$.
\end{definition}

Following \cite[Def. 1]{CMdef} we define:
\begin{definition}\label{CMdefLR}
 Let $M$ a manifold, $E\to M$ a vector bundle,
 $\cM$ a fixed locally finitely generated $C^{\infty}(M)$-submodule of $\Gamma(E)$ and $I=(-\epsilon,\epsilon)$ an interval.
\begin{enumerate}
 \item A family of Lie-Rinehart algebras over $I$ is a smoothly varying family $(\cA_t)_{t\in I} = (\cM,[\cdot,\cdot]_t,\rho_t)$.
 \item The families $(\cA_t)_{t \in I}$ and $(\cA'_t)_{t \in I}$ are equivalent if there exists a family of Lie-Rinehart algebra isomorphisms $h_t : \cA_t \to \cA'_t$ varying smoothly with respect to $t$.
 \item A deformation of a Lie-Rinehart algebra $\cA=(\cM,\rho,[\cdot,\cdot])$ is a family $(\cA_t)_{t \in I}$ of Lie-Rinehart algebras such that $\cA_0 = \cA$.
 \item   Two deformations $(\cA_t)_{t \in I}$ and $(\cA'_t)_{t \in I}$ of $(\cA,\rho,[\cdot,\cdot])$ are equivalent if there exists an equivalence $h_t : \cA_t \to \cA'_t$ with $h_0 = id$.
\end{enumerate}
\end{definition}

Similarly to the regular case, a singular foliation $\cF \subset \vX_c(M)$ can be equivalently described by saying that $\widehat{\cF}$ has the structure of a Lie-Rinehart algebra $(\cM,[\cdot,\cdot],\rho)$ over $C^{\infty}(M)$ for which $\rho \colon \cM \to Der(C^{\infty}(M))=\vX(M)$ is injective.  
In the following we work with $\widehat{\cF}$ (as introduced in Def. \ref{def:folhat}) rather than with $\cF$, in order to parallel the known  results   for regular foliations.


Denote by $\widehat{\cF}_{mod}$ the $C^{\infty}(M)$-module underlying  $\widehat{\cF}$ (i.e., $\widehat{\cF}_{mod}$ is $\widehat{\cF}$ seen simply as $C^{\infty}(M)$-module).  $\widehat{\cF}_{mod}$ is a locally finitely generated $C^{\infty}(M)$-submodule of $\vX(M)$.
Applying Def. \ref{CMdefLR}, we obtain that a \emph{deformation of a foliation} $(M,\cF)$ consists of Lie brackets 
$[\cdot,\cdot]_t$ and maps $\rho_t$, agreeing with the Lie bracket of vector fields $[\cdot,\cdot]_{\vX(M)}$ and the inclusion into $\vX(M)$ for $t=0$,
 such that $$(\widehat{\cF},[\cdot,\cdot]_t,\rho_t)$$ is a family of 
Lie-Rinehart algebras.
Notice that for every $t$ the module $\{\rho_t(X): X\in \widehat{\cF}\}$ corresponds to a foliation, and $\rho_t$ (as long as it is injective) determines $[\cdot,\cdot]_t$ by virtue of  $\rho_t([X,Y]_t)=[\rho_t(X),\rho_t(Y)]_{\vX(M)}$.

\begin{ep}
 Consider a foliation $(M,\cF)$ and let $\{\varphi^t\}_{t\in I}$ be a one-parameter family of diffeomorphisms of $M$. Denote by $[\cdot,\cdot]$ the Lie bracket of vector fields in $\cF$.
Then $(\widehat{\cF}_{mod},[\cdot,\cdot],(\varphi^t)_*)$ is a deformation as above.  
Further it is equivalent to the constant deformation (i.e., it is a trivial deformation), an equivalence being given by  $\{(\varphi^t)_*\}_{t\in I}$.
All trivial deformations of $(M,\cF)$ are of this kind, as a consequence of the fact that any algebra automorphism of $C^{\infty}(M)$ is the pullback of functions  by a diffeomorphism of $M$.
\end{ep}

The fact that in our notion of deformation of a foliation $(M,\cF)$ 
we insist on fixing the $C^{\infty}(M)$-module structure   makes our notion of deformation somewhat restrictive. For instance,
if two  foliations  $\cF_1,\cF_2$ on $M$ are isomorphic as 
$C^{\infty}(M)$-modules, then there exists a diffeomorphism $\phi$ of $M$ such that for all $x\in M$ the vector spaces $(\cF_1)_x$ and  $(\cF_2)_{\phi(x)}$ are isomorphic.

\subsubsection{Deformation cocycles}\label{subs:dc}

In \cite{Heitsch} deformations of regular foliations were shown to be controlled by ``foliated cohomology'', whereas in \cite{CMdef} by ``deformation cohomology''. Let us explain how these cohomologies are defined in our framework, i.e. when $\cF$ is a singular foliation:
\begin{itemize}
\item[1)] $H_{def}^*(\widehat{\cF})$ is the \emph{deformation cohomology} defined exactly as in the beginning of  \cite[\S 2]{CMdef}, replacing the $C^{\infty}(M)$-module of Lie algebroid sections with $\widehat{\cF}_{mod}$ (notice that the definition given there holds for any Lie-Rinehart $C^{\infty}(M)$-algebra.). An $n$-cochain is a multilinear and antisymmetric map $D : \underbrace{\widehat{\cF} \otimes \dots \otimes \widehat{\cF}}_\text{$n$-times} \to \widehat{\cF}$ which is a multiderivation. Namely, there exists a $C^{\infty}(M)$-multilinear map $\sigma_D : \underbrace{\widehat{\cF} \otimes \dots \otimes \widehat{\cF}}_\text{$n-1$-times} \to \vX(M)$, called the \textit{symbol} of $D$, such that $$D(X_1,\dots,fX_n) = fD(X_1,\dots,X_n) + \sigma_D(X_1,\dots,X_{n-1})(f)X_n$$ for all $X_1,\dots,X_n \in \widehat{\cF}$ and $f \in C^{\infty}(M)$. Notice that the above expression determines $\sigma_D$ uniquely. The boundary map $\delta : C^n_{def}(\widehat{\cF}) \to C^{n+1}_{def}(\widehat{\cF})$ is given by a suitable Eilenberg-Maclane formula (see \cite[\S 2]{CMdef}).

\item[2)]
$H^*(\widehat{\cF};\cN)$  is the \emph{foliated cohomology} defined as in \cite[\S 1]{Heitsch}, but using the map $ \widehat{\cF}\times \cN \to \cN$ induced by the Lie bracket of vector fields (see \S \ref{subsection:linholrep}), which makes  $\cN$ into a representation of $\widehat{\cF}$. An $n$-cochain is a map of $C^{\infty}(M)$-modules $\underbrace{\widehat{\cF}\wedge \dots \wedge \widehat{\cF}}_\text{$n$-times} \to \cN$.

\item[3)] The two cohomologies are related via a canonical map  
\begin{equation} \label{eqn:defiso}
 H^n_{def}(\widehat{\cF}) \to H^{n-1}(\widehat{\cF};\cN)
\end{equation} 
defined as follows: Given a cochain $D \in C^{n}_{def}(\widehat{\cF})$ consider its symbol $\sigma_{D}$ which is an $\vX(M)$-valued map. The assignment $D \to \natural \circ \sigma_D$ where $\natural : \vX(M) \to \vX(M)/\widehat{\cF} = \cN$ is the quotient map, induces a well defined map in cohomology. When $\widehat{\cF}$ corresponds to a regular foliation it was shown in \cite{CMdef} that this map is an isomorphism. 

\end{itemize} 

The following result and its proof is  an immediate adaptation
to the context of foliations (or Lie-Rinehart $C^{\infty}(M)$-algebras) 
 of  Prop. 2 and   Thm. 2  of \cite{CMdef}.
\begin{itemize}
 \item[A)] Given a deformation of $\cF$, consider the cocycles in 
 $C^{2}_{def}(\widehat{\cF})$ given by

 $$c_t(a,b) = \frac{d}{ds}|_{s=t}([a,b]_s)$$ for all $t$, and assume that $M$ is compact. Then the deformation is trivial (that is, equivalent to the constant deformation)
 if{f} there is a smooth family  $D_t\in C^{1}_{def}(\widehat{\cF})$ with $\delta(D_t)=c_t$. 
\end{itemize}

  When $\cF$ is regular, say $\cF=C^{\infty}_c(M,F)$ for an involutive regular distribution $F$ on $M$, and $F_t\subset TM$ is a smooth deformation of   $F$ by  involutive regular distributions, then Heitsch \cite{Heitsch} gave a geometric construction of  a class in $H^1(\widehat{\cF};\cN)$, using orthogonal projections $TM\to F_t$. In \cite{CMdef} it was shown that Heitsch's class corresponds to the class $[c_0] \in H^2_{def}(\widehat{\cF})$ via the canonical isomorphism 
(\ref{eqn:defiso}).

Given a deformation of a singular foliation now, it is impossible to construct geometrically a class in $H^1(\widehat{\cF};\cN)$ as in \cite{Heitsch} (the orthogonal projections used in \cite{Heitsch} are no longer smooth in the singular case). We do have the deformation class $[c_0] \in H^2_{def}(\widehat{\cF})$ defined in A) though, which we can carry to a class in $H^1(\widehat{\cF};\cN)$ using the map (\ref{eqn:defiso}). It would be interesting to investigate what this class is the obstruction to.
In the regular case it is exactly Heitsch's class \cite{Heitsch}. 

\appendix

\section{Appendix}

\subsection{Proofs of Theorem \ref{globalaction} and Proposition \ref{globalactionong}
}\label{holproofs}

Let $(M,\cF)$ be a foliation. 
If $W\subset \R^n\times M$ is a minimal bi-submersion obtained from a basis of $\cF_x$ and $b \colon M_0 \to \R^n$ a bisection defined in a neighborhood of $x$ in $M$, then
in general it is not true that $t b$ is a bisection at $x$ for all $t \in [0,1]$, as the following example shows.
\begin{ep}\label{notb}
If $M=\R$ and $\cF=\langle \partial_{x} \rangle$, then $b= -2x \colon M \to \R$ is a bisection. For any $t$, the map $\bt(t b) \colon M \to M$ is $x(1-2t)$, so  for $t=\frac{1}{2}$ it is not a diffeomorphism. Even more is true: since the diffeomorphism carried by $b$ is orientation-reversing, it is not isotopic to the identity.
\end{ep}

To prove Theorem \ref{globalaction} we will need bisections $b$ such that $t b$ is a bisection at $x$ for all $t \in [0,1]$. Their existence is guaranteed by the following two lemmas.

\begin{lemma}\label{epsb}
Let $x\in(M,\cF)$, $\{X_i\}_{i\le n}$ generators for $\cF$ in a small neighborhood $M_0$ of $x$ defining a basis of 
$\cF_x$.
Let $W\subset \R^{n}\times M_0$ be the corresponding bi-submersions, and 
denote by  $L$ the leaf through $x$. 
 
If $b$ is a section of $\bs \colon M_0 \to W$  such that $(d_xb)(T_xL)=\{0\}\times T_xL$ then the image of $b$ is a bisection near $x$.
Further, $\emph{image}(tb)$ is a bisection at $x$ for all $t \in [0,1]$.
\end{lemma}
\begin{proof} Let $y:=\bt(b(x))\in L$.
Clearly the source $\bs$ maps $\bt^{-1}(y)$ into $L$. Hence $d_{b(x)}\bs$ maps $T_{b(x)}\bt^{-1}(y)$ into $T_xL$. So
 \begin{align*}
&T_{b(x)}(\text{image}(b))\cap T_{b(x)}\bt^{-1}(y)\\
=&T_{b(x)}(\text{image}(b))\cap T_{b(x)}\bt^{-1}(y)\cap (d_{b(x)}\bs)^{-1}(T_xL)\\
=&T_{b(x)}(\text{image}(b|_L))\cap T_{b(x)}\bt^{-1}(y)\\
=&\{0\}.
\end{align*}
The last equality holds for the following reason: By our assumption $T_{b(x)}(\text{image}(b|_L))= \{0\}\times T_xL=T_{b(x)}(\text{graph}(\hat{b}|_L))$ where $\hat{b} \colon  M_0 \to \R^n$ is the constant map with value $b(x)$. The latter   defines a bisection as $\bt \circ (\hat{b}\times Id_{M_0})$ is a diffeomorphism, namely $exp(\sum_{i \le n} b(x)_i X_i)$.

  The last statement of the lemma holds simply because $d_x(t b)(T_xL)=
\{0\}\times T_xL$.\end{proof}
   
\begin{remark} 
We make a brief comment about the composition of path-holonomy bi-submersions. 
For $\alpha=1,2$ let $x^{\alpha}\in M$,    $\{X_i^{\alpha}\}_{i\le n}$ generators for $\cF$ defining a basis of 
$\cF_{x^{\alpha}}$, and denote by $W^{\alpha}$   the corresponding bi-submersions. By \cite[Prop. 2.10 a)]{AndrSk}  $W^{\alpha}$ is a neighborhood of $(0,x^{\alpha})$ in  $\R^{n}\times M$. The composition\footnote{Notice that this might be empty, for instance if the supports of the $\{X_i^{2}\}$ are disjoint from the support of the $\{X_i^{1}\}$.}
 of the two bi-submersions \cite[Prop. 2.4  b))]{AndrSk} is $$W^2 \circ W^1=
\{(\lambda^2,y^2),(\lambda^1,y^1):exp_{y^1}(\sum \lambda_i^1 X_i^1)=y^2, (\lambda^{\alpha},y^{\alpha})\in W^{\alpha}\}\subset (\R^n\times M)\times (\R^n \times M).
$$
As the component $y^2$ is determined by the other three, $W_2\circ W_1$
  can be identified canonically with $$\{(\lambda^2,\lambda^1,y^1):(\lambda^1,y^1)\in W^1, (\lambda^2, exp_{y^1}(\sum \lambda_i^1 X_i^1))\in W^2\}\subset
\R^n  \times \R^n \times M.$$ In the following we will freely switch from viewing an $N$-fold composition of minimal bi-submersions as an open subset of $(\R^n \times M) \times\cdots \times (\R^n \times M)$ to viewing it as an open subset of 
$\R^n \times\cdots \times \R^n \times M$.
\end{remark}

\begin{lemma}\label{epsbmany}
Let $W$ be any bi-submersion in the path-holonomy atlas   containing the point $x$, and $M_0$ a neighborhood of $x$ in $M$. 
Then $W$ is isomorphic to a finite composition of path-holonomy bi-submersions, so we may assume that $W\subset \R^n\times\cdots \times \R^n\times M$ is such a composition.

The graph  of any map $b \colon M_0 \to \R^n\times\cdots \times \R^n$
satisfying $(d_xb)(T_xL)=0$ can be canonically deformed to the   bisection $\{0\}\times M_0$ by a path of bisections of $W$.
\end{lemma}
\begin{proof} Every bi-submersion  $W$ in the path-holonomy atlas    is by definition the composition of path-holonomy bi-submersions and their inverses. If $U=\R^n\times M \rightrightarrows M$ is the path-holonomy bi-submersion defined by a choice of local generators $X_1,\dots,X_n$ of $\cF$, then its inverse $U^{-1}$ is obtained by switching the roles of $\bs$ and $\bt$, and the map $U\to U^{-1}, ( {\lambda},y)\to (-\lambda,exp_y(\sum \lambda_i X_i))$
provides an isomorphism of bi-submersions. This proves the first statement of the lemma.


Assume now that the bi-submersion $W$ is a composition $W=W_N\circ \cdots\circ W_1$ where  
$W_{\alpha}\subset \R^{n}\times M$ is a minimal   bi-submersion ($\alpha=1,\dots,N$).
If $b=(b^{N},\dots,b^{1})\colon M_0 \to \R^n\times\cdots \times \R^n$ satisfies $(d_xb)(T_xL)=0$, then the path  defined on $[0,N]$ by
\begin{align}\label{bt}
&b_{t}:=(0,0, \dots,0,0, tb^1) \text{ for } t \in [0,1],\\ 
&b_{t}:=(0,0,\cdots,0,(t-1)b^2,b^1) \text{ for } t \in [1,2], \nonumber\\
&\cdots \nonumber\\
&b_{t}:=((t-N+1)b^N,b^{N-1},\cdots, b^2 ,b^1) \text{ for } t \in [N-1,N] \nonumber
\end{align}
consists of maps whose graphs are bisections of $W$. This is seen applying repeatedly  Lemma \ref{epsb}. Notice that when $t=0$ we obtain the zero map, and when $t=N$ we obtain the map $b$.   \end{proof}

\begin{lemma}\label{bisexists}
Let $U$ be any bi-submersion, $p\in U$. Let $S_x$ and $S_y$ be slices  transverse to the foliation at $x:=\bs(p)$ and $y:=\bt(p)$ respectively. Then there exists a bisection of $U$ through $p$ whose corresponding diffeomorphism maps $S_x$ into $S_y$.
\end{lemma}
\begin{proof}
\emph{{Claim}: The map $(\bs,\bt) \colon U \to M \times M$ is transverse at $p$
to the submanifold $S_x \times S_y$.}

We have to show that $Im(d_p(\bs,\bt))+(T_xS_x \times T_yS_y)=T_{x,y}(M\times M)$.
This follows from the fact that $(T_xL\times T_yL)\subset Im(d_p(\bs,\bt))$, which in turn holds for the following reason: Given $X\in T_xL, Y\in T_yL$, since $d_p\bt(Ker(d_p\bs))=T_yL$ by \cite[Def. 2.1 ii)]{AndrSk},
 there exists $Z'\in Ker(d_p\bs)$ with $d_p \bt(Z')=Y$. Similarly,  there exists $Z''\in Ker(d_p\bt)$ with $d_p\bs(Z'')=X$. Hence $Z'+Z''$ maps to $(X,Y)$ under $d_p(\bs,\bt)$. 
 
From the above claim it follows that, shrinking the transversals if necessary and working in a neighborhood of $p$, $C:=\bs^{-1}(S_x)\cap \bt^{-1}(S_y)$ is a submanifold of $U$ of codimension $2dim(L)$.
\bigbreak
\emph{{Claim}: Denote by ${\sigma} \colon C  \to S_x$ the restriction of $\bs$. Then $d_p\sigma$ is surjective with kernel $Ker(d_p \bs) \cap Ker(d_p \bt)$}.

Let $Z\in ker (d_p\sigma)$. In particular $Z \in ker (d_p \bs)$, so  $d_p \bt(Z)\in T_yL$. Since at the same time $d_p \bt(Z)\in T_yS_y$ we conclude that $Z\in ker (d_p \bt)$. So $ker (d_p\sigma)\subset Ker(d_p \bs) \cap Ker(d_p \bt)$. The latter is the kernel of the surjective linear map
  $d_p\bt \colon Ker(d_p \bs) \to T_yL$, hence it has dimension $dim(U)-dim(M) -dim(L)$.
Hence $$dim(Im(d_p\sigma))=dim(C)-dim(ker(d_p\sigma))\ge dim(M)-dim(L)=dim(S_x),$$ proving both statements of  the claim. 

The last claim has several consequences. As $\sigma$ is a submersion near $p$,  we can find a $\beta \colon S_x \to \bs^{-1}(S_x)\cap \bt^{-1}(S_y)$ with $\sigma \circ \beta=Id_{S_x}$. 
Further, the intersection of $Ker(d_p \bt)$ with $T_pC$ is $Ker(d_p \bs) \cap Ker(d_p \bt)$. As the image of $\beta$ is transversal to the latter, it follows that  $\bt \circ\beta \colon S_x \to S_y$ is a diffeomorphism onto its image. In particular, we can find an 
 extension of $\beta$ to a section $b$ of $\bs$ such that $T_p\text({image}(b))$  has trivial intersection with $Ker(d_p \bt)$, so $b$ will be a bisection of $U$.
\end{proof}

\begin{lemma}\label{vfinS}
Let $x\in (M,\cF)$, $S_x$ a slice at $x$ transverse to the foliation, and $\{Z_t\}_{t\in[0,1]}$ a time-dependent vector field on $M$ lying in  $I_x\cF$ whose time-1 flow $\psi$ satisfies $\psi(S_x)\subset S_x$. Then there exits a 
time-dependent vector field on $S_x$, lying in $I_x\cF_{S_x}$, whose time-1 flow equals $\psi|_{S_x}$.
\end{lemma}

\begin{proof}

Choose generators of $\cF$ in a neighborhood $W$ of $x$ as in Prop. \ref{thm:splitting} (the splitting theorem), that is: the first elements $X_1,\dots,X_{k}$ restrict to  coordinate vector fields on $L$ and commute with all $X_i$'s, and the remaining  
elements $X_{k+1},\dots,X_n$ restrict to generators of $\cF_{S_x}$. 
Denote $pr \colon W\cong I^k\times S_x\to S_x$ the projection onto the second factor, where the isomorphism is given as in Prop. \ref{thm:splitting}. Denote by $\psi_t$ the time-$t$ flow of $\{Z_t\}_{t\in[0,1]}$.
 
The restriction $pr|_{\psi_t(S_x)}\colon \psi_t(S_x)\to S_x$ is a diffeomorphism onto its image:  the diffeomorphism $\psi$ fixes $x$ and preserves the leaf $L$, so its derivative preserves the tangent space to the leaf $T_xL$, and therefore 
$d_x\psi(T_xS_x)=T_x(\psi_t(S_x))$ has trivial intersection with $T_xL=ker(d_x pr)$. We define $$\phi_t:=pr\circ \psi_t|_{S_x} \colon S_x\to S_x,$$
which is also a diffeomorphism onto its image. Notice that $\phi_0=Id_{S_x}$ and $\phi_1=\psi|_{S_x}$. Consider the 
time-dependent vector field $\{Y_t\}_{t\in[0,1]}$ on $S_x$ whose flow is $\{\phi_t\}_{t\in[0,1]}$. We just have to show that, for every fixed  $t$, $Y_t\in I_x\cF_{S_x}$.

To this aim we fix $p\in S_x$ and compute
$$Y_t(\phi_t(p))=\frac{d}{dt}(\phi_t(p))=pr_*\frac{d}{dt}(\psi_t(p))=pr_*(Z_t(\psi_t(p))).$$
Writing $Z_t=\sum_{i=1}^n f_i^t\cdot X_i$, with $f_i^t\in C^{\infty}(W)$ vanishing at the point $x$, we see that
$$pr_*(Z_t(\psi_t(p)))=\left[\sum_{i=k+1}^n 
(pr|_{\psi_t(S_x)})^{-1})^*
f_i^t\cdot X_i\right](\phi_t(p)).$$
(Here we used that $pr_*(X_i)$ equals zero for $i\le k$ and that it equals ${X_i}|_{S_x}$ for $i> k$.) As $p$ is arbitrary, we conclude that the vector field $Y_t$ equals the expression in the square bracket. The latter clearly lies in $I_x\cF_{S_x}$, for ${X_i}|_{S_x}\in \cF_{S_x}$  and 
$f_i^t$ vanishes at $pr(x)=x$.
\end{proof}

\begin{proof}[\un{Proof of Theorem \ref{globalaction}}]
We  show that the map $\Phi_x^y$ is well-defined, i.e. that it does not depend on the choices of bi-submersion $U$ and of map $\bar{b} \colon S_x \to U$.

Fix a point $h$ of $H$, let $x:=\bs(h),y:=\bt(h)$. Let $U$ be a bi-submersion in the path-holonomy atlas, let $u\in U$ such that $[u]=h$,
and $\bar{b} \colon S_x \to U$ a section of $\bs$ through $u$ such that $\bt \circ \bar{b}$ maps  $S_x$ to $S_y$ (it exists as a consequence of Lemma \ref{bisexists}). We may assume that $U$ is a composition of $N$ path-holonomy bi-submersions (see Lemma \ref{epsbmany}).  
Take $\bar{b}$ and extend it to a bisection $b$ of $U$ such that 
$(d_xb)(T_xL)=\{0\}\times T_xL$, where $L$ denotes the leaf through $x$. By Lemma  \ref{epsbmany}, the bisection $b$ can be deformed canonically to the zero-bisection by a path of bisections $b_t$.

Similarly, choose another bi-submersion $U'$ in the path-holonomy atlas, a point $u'$ with $[u']=h$ and 
 $\bar{b}'\colon S_x \to U'$ a section  of $\bs'$ through $u'$ such that $\bt' \circ \bar{b}'$ maps  $S_x$ to $S_y$. 
Since $u$ and $u'$ represent the same point $h\in H$, by definition there exists a morphism of bi-submersions $\sigma \colon U' \to U$ with $u' \mapsto u$. Then $\sigma \circ  \bar{b}'\colon S_x \to U$ is a section of $\bs$ through $u$ carrying the same diffeomorphism as $\bar{b}'$.
Take ${\sigma \circ  \bar{b}'}$ and extend it to a bisection $b'$ of $U$, defined in some neighborhood $M_0$ of $x$,  such that 
$(d_xb')(T_xL)=\{0\}\times T_xL$.  By Lemma  \ref{epsbmany}, $b'$
is a bisection of $U$ which can be deformed canonically to the zero-bisection by a path of bisections  $b'_t$. 
 %

Our aim is to compare the diffeomorphisms $S_x \to S_y$ induced by $\bar{b}=b|_{S_x}$ and by $\bar{b}'=b'|_{S_x}$. Denote by $\phi_{t}$  (resp. $\phi'_{t}$) the local diffeomorphisms of $M$ carried by 
$b_t$ (resp $b'_t$), for $t\in [0,N]$. 

\emph{{Claim}: $\{\phi_t^{-1}\circ \phi'_t\}_{t\in[0,N]}$ is the   flow of a time dependent vector field that lies in $I_{x}\cF$.}

To prove the claim we proceed as follows. Recall from Lemma \ref{epsbmany} that $U$ is a product of path-holonomy bi-submersions: $U=W^{N}\circ\cdots\circ W^1$. Define $\{Z_t\}$ to be the time-dependent vector field corresponding to the 1-parameter family of diffeomorphisms $\{\phi_t\}$. Fix $t\in [0,N]$, and $\alpha\in\{1,\dots,N\}$ so that  $t\in [\alpha-1,\alpha]$. To write down explicitly $Z_t$ , denote by $\{X^{\alpha}_i\}_{i\le n}$ the  vector fields in $\cF$ that give rise to $W^{\alpha}$. Using eq. \eqref{bt}, we have for all $z\in M_0$:
$$Z_t(\phi_t(z))=\frac{d}{dt}\phi_t(z)=\sum_i b_i^{\alpha}(z) \cdot X_i^{\alpha}(\phi_t(z))  
,$$ or equivalently\footnote{If $b_t$ was not a bisection,  $\phi_t$ would not be invertible and we could not define the vector field $Z_t$.}
\begin{equation}\label{Zt}
Z_t=\sum_i((\phi_t^{-1})^*b_i^{\alpha})\cdot X_i^{\alpha}.
\end{equation} 

Similarly, the  time-dependent vector field corresponding to $\phi'_t$ is
$Z'_t= \sum_i(({\phi'}_t^{-1})^*{b'_i}^{\alpha})\cdot X_i^\alpha$.
Now eq. \eqref{Posieq}, which gives a relation between the flows of 
 any two time-dependent vector fields,  implies  
\begin{equation}\label{sortaBCHapp}
\phi'_t=\phi_t\circ (\text{time-}t \text{ flow of }\{(\phi_s)^{-1}_*(Z'_s-Z_s)\}_{s\in [0,N]}).
\end{equation}
By eq. \eqref{Zt}, for all $s \in [\beta-1,\beta]$ where $\beta\in\{1,\dots,N\}$, we have 
$Z'_s-Z_s=\sum_ig_{i,s}X_i^{\beta}$
where the functions $g_{i,s}$ (defined in a neighborhood of $\phi_s(x)$) are given by
$$g_{i,s}:= {(\phi_s^{-1})^*b_i^{\beta}- ({\phi'}_s^{-1})}^*{b'}_i^{\beta}.$$
Notice that $g_{i,s}$ vanishes at $\phi_s(x)$ as a consequence of the fact that
$b(x)=b'(x)=u$ and
 $\phi_s(x)=\phi'_s(x)$. Equivalently, $\phi_s^* g_{i,s}\in I_x$. Hence
$$(\phi_s)^{-1}_*(Z'_s-Z_s)=\sum_i(\phi_s^* g_{i,s})\cdot(\phi_s)^{-1}_*X_i^\beta\in I_{x}\cF.$$
From eq. \eqref{sortaBCHapp} it follows that $\phi_t^{-1}\circ \phi'_t$ is the time-$t$ flow of a time dependent vector field that lies in $I_{x}\cF$, proving the claim. 

Finally, since $\phi_N(S_x)=(\bt\circ b)(S_x)\subset S_y$ and similarly for $\phi_N'$, we can apply Lemma \ref{vfinS} to conclude the proof.
\end{proof}

For the proof of Prop. \ref{globalactionong} we need one more lemma:

\begin{lemma}\label{ypsiy}
Let $x\in (M,\cF)$, $Y\in \cF(x)$ and $\psi\in exp(I_x\cF)$.
Then $Y - \psi_*Y\in I_x \cF$.
\end{lemma}
\begin{proof}   For any  time-dependent vector field $X$ whose time $1$ flow is $\psi$,
we have
\begin{equation}\label{eq:ypsiy}
 Y-\psi_*Y=\psi_*(\psi^{-1}_*Y-Y)=\psi_*\left[\int_0^1 \frac{d}{dt} ((\psi^{-1}_t)_*Y) dt\right]=\psi_*\left[\int_0^1 (\psi^{-1}_t)_*\left([X_t,Y]\right) dt\right],
\end{equation}
where $\psi_t$ is the time-$t$ flow of $\{X_t\}$.
The last equation holds because the integrands are equal,  
see for example \cite{Lee}. 

 Since $\psi \in exp(I_x \cF)$, we can choose $X$ so that $X_t\in I_x \cF$ for all $t$.   
From this and $Y(x)=0$, using the Leibniz rule one shows that $[X_t,Y]\in I_x\cF$. Since $\psi_t(x)=x$ for all $t$, we are done. 
 \end{proof}

\begin{proof}[\un{Proof of Prop. \ref{globalactionong}}] Fix $h\in H_x^y$ and slices $S_x$ and $S_y$.
Any  diffeomorphism $\tau$ chosen as in Theorem \ref{globalaction} maps $x$ to $y$ and maps the foliation $\cF_{S_x}$ to $\cF_{S_y}$, hence gives a  map 
 $\g_x\to\g_y$ as in eq. \eqref{gxgymap}.
 
First we show that the map  \eqref{gxgymap} is independent of the choice of  the diffeomorphism $\tau$.  
Let $\tau'$ be another  diffeomorphism
 associated to $h$ as in Theorem \ref{globalaction}.
 Given $Y\in \cF_{S_x}$, we have to show that $\tau_*Y-\tau'_*Y \in I_y\cF_{S_y}$.  This goes as follows: we have $$\tau_*Y-\tau'_*Y=\tau_*(Y - \psi_*Y)$$
where $\psi:=\tau_*^{-1}\circ \tau'_* \in exp(I_x \cF_{S_x})$ by    Theorem \ref{globalaction}. Hence Lemma \ref{ypsiy}, applied to the foliation $(S_x,\cF_{S_x})$, implies that 
$Y - \psi_*Y\in I_x \cF_{S_x}$,  and with $\tau_*(I_x \cF_{S_x})= I_y\cF_{S_y}$ we are done.

Last, we show that the map \eqref{gxgymap} is independent of the choice of slices, and hence canonical. Denote $S^1_x:=S_x$ and let  $S^2_x$ be
another slice at $x$. Fix  $\psi^{12}_x \in exp(I_x\cF)$ mapping $S^2_x $ to $S^1_x$. Notice that the isomorphism 
\begin{equation}\label{S21}
\cF_{S^2_x}/I_x\cF_{S^2_x}\to \cF_{S^1_x}/I_x\cF_{S^1_x}
\end{equation}
induced by   $(\psi^{12}_x)_*$ (or, more precisely, by its restriction 
to the slice $S_x^2$) coincides with the one
obtained by the canonical identification of both sides of eq. \eqref{S21}  with $\g_x$. This is a consequence of the fact that the automorphism of $\g_x$ induced by $(\psi^{12}_x)_*$ is $Id_{\g_x}$, by Lemma \ref{ypsiy}. Similarly, denote $S^1_y:=S_y$ and  let $S^2_y$ be
another slice at $y$, and $\psi^{21}_x \in exp(I_y\cF)$ mapping $S^1_y $ to $S^2_y$. We have to show that the diagram of isomorphisms
\begin{equation*}
\xymatrix{
 \cF_{S^2_x}/I_x\cF_{S^2_x}\ar[d]\ar[r] & \cF_{S^2_y}/I_y\cF_{S^2_y} 
 \\
\cF_{S^1_x}/I_x\cF_{S^1_x} \ar[r] & \cF_{S^1_y}/I_y\cF_{S^1_y} \ar[u]\\
}
\end{equation*}
commutes, where the horizontal maps are given by \eqref{gxgymap} (applied to the two choices of slices)
and the vertical maps are induced by $(\psi^{12}_x)_*$ and $(\psi^{21}_y)_*$ respectively, as in eq. \eqref{S21}. 
We have $$\Phi_x^y (h) =\langle \tau \rangle \in  GermAut_{\cF}(S_x^1,S_y^1)/exp(I_x \cF_{S^1_x}),$$
hence by Lemma \ref{immaterial} we have
$$\Phi_x^y (h) =\langle \psi^{21}_y \circ\tau \circ \psi^{12}_x \rangle \in  GermAut_{\cF}(S_x^2,S_y^2)/exp(I_x \cF_{S^2_x}),$$ 
showing that the above diagram commutes.  \end{proof}

\subsection{Changing transversals}\label{subsec:chan}

We put here some technical results regarding different choices of transversals which are used  in \S \ref{section:geomhol}.

\begin{lemma}\label{s1s2} Let $x$ be a point in a foliated manifold $(M,\cF)$, 
and $S^1_x,S^2_x$ two transversals at $x$. Then there exists $\psi \in exp(I_x\cF)$ mapping $S^1_x $ to $S^2_x$.
\end{lemma}
\begin{proof}
By the splitting theorem Prop. \ref{thm:splitting} there exists 
a neighborhood $W$ of $x$ in $M$ and a diffeomorphism  of foliated manifolds   that identifies $(W,\cF_W)$ with the
product of the foliated manifolds $(S_x^1, \cF_{S_x^1}) $ and $I ^k$ endowed with the foliation consisting of just one leaf, where $I:=(-1,1)$. We use this identification and denote by $\pi \colon  S_x^1 \times I ^k \to S^1_x$ the natural projection. As $\pi^{-1}(x,0)$ is given by the leaf of $\cF$ through $x$, it is clear that $\pi$ maps $S^2_x$ diffeomorphically onto $S^1_x$. Hence there is a unique   map $\theta \colon S^1_x \to I ^k$ 
whose graph is $S^2_x$. Denote by $(s_1,\dots,s_k)$ the standard coordinates on $I ^k$, and consider the vector field $\sum_{i=1}^k \pi^* \theta_i \cdot \partial_{s_i}$ on  $S_x^1 \times I ^k$, which  lies in $I_x \cF$ since $\theta(x)=0$. Its time-one flow  takes $S^1_x$ to  $S^2_x$.
\end{proof}

\begin{lemma}\label{ideslices} Let $x,y$ be  points in a foliated manifold $(M,\cF)$ lying in the same leaf and $S^i_x, S^i_y$ transversals at $x$ and $y$ respectively, $i=1,2$. There is an  identification 
\begin{align}\label{identtransversals}
GermAut_{\cF}(S_x^1,S_y^1)/exp(I_x \cF_{S^1_x})  &\to GermAut_{\cF}(S_x^2,S_y^2)/exp(I_x \cF_{S^2_x}) , \quad
[\tau] &\mapsto [\tilde{\psi}_y^{21} \circ \tau \circ \tilde{\psi}^{12}_x]
\end{align}
where  $\psi^{12}_x \in exp(I_x\cF)$ maps $S^2_x $ to $S^1_x$, $\psi^{21}_y \in exp(I_x\cF)$ maps $S^1_y $ to $S^2_y$, and $\tilde{\psi}^{12}_x$ (resp. $\tilde{\psi}_y^{21}$) is the restriction to $S^2_x$ (resp. ${S^1_y}$). 

 Further, this identification is canonical, namely independent of the choice of $\psi^{12}_x$ and $\psi^{21}_y$. 
\end{lemma}
\begin{proof}
Make a choice for  $\psi^{12}_x$   and $\psi^{21}_y$ (this is possible by Lemma \ref{s1s2}). To show that the map \eqref{identtransversals} is well-defined, we take $\tau,\hat{\tau} \in GermAut_{\cF}(S_x^1,S_y^1)$ with $\hat{\tau}^{-1} \circ \tau \in exp(I_x \cF_{S^1_x})$ and need to check that
$$(\tilde{\psi}_y^{21} \circ \hat{\tau} \circ \tilde{\psi}^{12}_x)^{-1}\circ (\tilde{\psi}_y^{21} \circ \tau \circ \tilde{\psi}^{12}_x)
\in exp(I_x \cF_{S^2_x}).$$
This is done using that for any time-dependent vector field $X$ in $\cF$ and $f\in I_x$, we have
$$(\psi^{12}_x)^{-1} \circ exp(fX) \circ \psi^{12}_x=exp ((\psi^{12}_x)^{-1}_* (fX))\in exp(I_x \cF).$$

Now we show that the map is canonical: Let $\hat{\psi}^{12}_x$   and $\hat{\psi}^{21}_y$ be as above; we have to show that 
$$(\tilde{\psi}_y^{21} \circ \tau \circ \tilde{\psi}^{12}_x)^{-1}\circ (\tilde{\hat{\psi}}_y^{21} \circ \tau \circ  \tilde{\hat{\psi}}^{12}_x)\in exp(I_x \cF_{S^2_x}),$$
which follows from a computation similar to the above.  \end{proof}

\begin{lemma}\label{derid} Let $x$ be a point in a foliated manifold $(M,\cF)$, lying in the leaf $L$.
If $\psi \in exp(I_x\cF)$ then $d_x \psi$ induces the identity on $N_xL:=T_xM/T_xL$.
\end{lemma}
\begin{proof}
 Denote by $\{X_t\}\subset I_x\cF$ the time-dependent vector field whose time-1 flow is $\psi$. As $X_t$ is tangent to $L$ and vanishes at $x$, it is clear that $d_x \psi$ maps $T_xL$ to itself, so it induces an endomorphism of $N_xL$.
 
To show that this endomorphism is the identity we proceed as follows: Let $Y_x \in T_xM$, and extend it to a vector field $Y$ defined near $x$. 
Since $X_t\in I_x\cF$, using the Leibniz rule one shows that $[X_t,Y]|_x\in T_xL$. Hence  $(\psi^{-1}_t)_*\left([X_t,Y]\right)|_x\in T_xL$ for all $t$, so from eq. \eqref{eq:ypsiy} we conclude that $Y_x-\psi_*Y_x \in T_xL$.  
\end{proof}

\subsection{The normal module \texorpdfstring{$\cN$}{Lg}}\label{subsec:normalmodule}

Let $(M,\cF)$ be a manifold with a foliation.
  Here we study the $C^{\infty}(M)$-module $\cN = \vX(M)/\widehat{\cF}$ which is needed in \S \ref{subsection:linholrep} and \S \ref{section:deform}. Notice that $\cN$  is locally finitely generated (since $\vX(M)$ is), and that it
 does not inherit the Lie bracket of $\vX(M)$. 

For every $x\in M$ consider the vector space
$$\cN_{x} := \cN / I_{x}\cN = \vX(M) / (\widehat{\cF} + I_{x}\vX(M)).$$
\begin{lemma}\label{NxL} 
Let $x$ belong to the leaf $L$. Then the evaluation map identifies $\cN_{x}$ with $N_{x}L = \frac{T_{x}M}{T_{x}L}$.
\end{lemma}
\begin{proof}
The map 
$$\vX(M) / (\widehat{\cF} + I_{x}\vX(M)) \to N_xL,\;\;\; \langle X\rangle \mapsto X_x \text{ mod }T_xL$$
is clearly well-defined and surjective. It is injective because if $X\in \vX(M)$ satisfies $X_x\in T_xL$, then there exists $Y\in \widehat{\cF}$ with $Y_x=X_x$, and hence $X-Y\in I_{x}\vX(M)$. 
\end{proof}

The union $N = \cup_{x \in M}\cN_x$ should be considered the normal bundle of the foliation $\cF$. It is a field of vector spaces of varying dimensions over $M$. 
 Given an embedded leaf $L$ we now interpret $\cN / I_{L}\cN$, which is the same as $\vX(M) / (\widehat{\cF} + I_{L}\vX(M))$. 
 \begin{lemma}\label{lem:NL} Let $L$ be an embedded\footnote{If the leaf $L$ is  not embedded, then in order to realize $C^{\infty}(L;NL)$ from $\cN$
one can proceed as in \cite[Rem. 1.16]{AndrSk} or \cite[Rem. 1.8]{AnZa11}, replacing $\cF$ there by $\cN$.} leaf of $(M,\cF)$, and let 
$NL:=\frac{T_{L}M}{TL}$ be its normal bundle.
The evaluation map identifies  $\cN / I_{L}\cN$ with $C^{\infty}(L;NL)$. Hence $\cN / I_{L}\cN$ is a projective $C^{\infty}(L)$-module.
\end{lemma}
\begin{proof}
We parallel  the proof of Lemma \ref{NxL}. The map $$q \colon   \vX(M) / (\widehat{\cF}  + I_{L}\vX(M))   \to C^{\infty}\left(L;\frac{T_{L}M}{TL}\right),\;\;\; \langle X \rangle \mapsto  X\mid_{L}  \text{ mod }TL$$ is well defined and surjective. 
We show  its injectivity. If $X\in \vX(M)$ satisfies $X|_L\subset TL$, then there exists $Y\in \widehat{\cF}$ with $Y_L=X_L$ (this is clear locally by Prop. \ref{thm:splitting}, and holds on the whole of $L$ by a partition of unity argument).
 Hence $X-Y\in I_L\vX(M)$. 
\end{proof}

\begin{remark}\label{rem:highordtransf}(\textbf{Higher order holonomy transformations})

The normal bundle $N$ carries the  first order transversal information of the foliation $(M,\cF)$, whence the correct way to think of the action of the holonomy groupoid $H$ on $N$ is as linearized holonomy transformations. However, higher order holonomy is interesting as well (\cf Dufour et al \cite{Dufour1, Dufour2}). 
In the regular case, this higher order holonomy is described by an action on $\cN$ of the $k$-th jet prolongation $J^k H$ of the (smooth) holonomy groupoid,    for every $k \in \N$ (\cf \cite[App. A, B]{Evens-Lu-Weinstein}). Roughly speaking, for any Lie groupoid $G$ its $k$-th jet prolongation  $J^k G$ is the Lie groupoid formed by the $k$-th tangent spaces of all its bisections. When $(M,\cF)$ is singular the holonomy groupoid is no longer smooth, so we cannot apply directly this definition. However, $J^k H$ can be defined as a quotient. Let us sketch this construction for $k=1$:

Given a bi-submersion $(U,\bt,\bs)$ of $(M,\cF)$, one can define its first jet prolongation $(J^1 U, j^1 \bt, j^1 \bs)$. Elements of $J^1 U$ are of the form $TV$, where $V$ is a bisection of $U$, and $j^1 \bt = \bt\circ\pi$, $j^1 \bs = \bs\circ\pi$, where $\pi : TV \to V$ is the bundle projection. One checks easily that it is also a bi-submersion of $(M,\cF)$. Furthermore, a morphism of bi-submersions $f : U_1 \to U_2$ can be prolonged to a morphism $j^1 f : J^1 U_1 \to J^2 U_2$. This way, given an atlas $\cU = \{(U_i,\bt_i,\bs_i)\}_{i \in I}$ we define $J^1 H$ as the quotient of the atlas $J^1\cU = \{(J^1 U_i, j^1 \bt_i, j^1 \bs_i)\}_{i \in I}$.

Now a bisection $V$ of $U$ corresponds to a local diffeomorphism $\phi^V$ of $M$ such that $(\phi^V)_* \cF \subset \cF$. It follows that the class of $TV$ in $J^1 H$ corresponds exactly to the germ of $(\phi^V)_* : \vX(M) \to \vX(M)$, which descends to $(\phi^V)_* : \cN \to \cN$. This formula defines the action of $J^1 H$ on $\cN$.
\end{remark}


\begin{thebibliography}{AA}


\bibitem[ABT]{ABT}	{\sc M. Alexandrino, R. Briquet, D. T\"oben} Progress in the Theory of Singular Riemannian Foliations.
{\it Differential Geometry and its Applications}
{\bf 31} (2013), 248--267.



\bibitem[AnSk06]{AndrSk} {\sc I. Androulidakis and G. Skandalis} The holonomy groupoid of a singular foliation. {\it J. Reine Angew. Math.} {\bf 626} (2009), 1--37.



\bibitem[AnZa11]{AnZa11} {\sc I. Androulidakis and M. Zambon} Smoothness of holonomy covers for singular foliations and essential isotropy. 
{\it Mathematische Zeitschrift}, to appear.
(Available as {arXiv:1111.1327}).




\bibitem[CrMa10]{CrMar} {\sc M. Crainic and I. Mar\c{c}ut} A Normal Form Theorem around Symplectic Leaves. {\it J. Differential Geom.}{\bf 92},  (2012), 417--461.




\bibitem[CrMo08]{CMdef} {\sc M. Crainic and I. Moerdijk} Deformations of Lie brackets: cohomological aspects. {\it Journal of the EMS} {\bf 10} (2008), 1037--1059.

\bibitem[CrStr11]{CrStr} {\sc M. Crainic and I. Struchiner} On the linearization theorem for proper Lie groupoids.
{\it Annales scientifiques de l'ENS}, to appear.
(Available as arXiv:1103.5245). 
 
\bibitem[Co79]{Connes} {\sc A. Connes} Sur la th\`{e}orie non commutative de l'int\`{e}gration. {\it Alg\`{e}bres d'op\'{e}rateurs (S\'{e}m., Les Plans-sur-Bex, 1978),} pp. 19--143, Lecture Notes in Math., 725, Springer, Berlin, 1979.






\bibitem[De13]{Debord}{\sc C. Debord} Longitudinal smoothness of the holonomy groupoid. {\it Comptes Rendus Mathematique}, Available online 3 September 2013, ISSN 1631-073X, \url{http://dx.doi.org/10.1016/j.crma.2013.07.025}.

\bibitem[DuWa06]{Dufour1}{\sc J.-P. Dufour and A. Wade.} Stability of higher order singular points of Poisson manifolds and Lie algebroids. {\it Ann. Inst. Fourier}, {\bf 56} (2006), fasc. 3, 545--559.

\bibitem[Du08]{Dufour2}{\sc J.-P. Dufour.} Examples of higher order stable singularities of Poisson structures. {\it Contemporary Mathematics} {\bf 450} (2008), 103--111.

\bibitem[EvLuWe99]{Evens-Lu-Weinstein} {\sc S. Evens, J.-H. Lu and A. Weinstein} Transverse measures, the modular class and a cohomology pairing for Lie algebroids. {\it Quart. J. Math. Oxford (2)} {\bf 50} (1999), 417--436.

\bibitem[Fe02]{Fernandes} {\sc R. L. Fernandes} Lie algebroids, Holonomy and Characteristic classes. {\it Adv. Math.} {\bf 170} (2002), 119--179.


\bibitem[He73]{Heitsch} {\sc J. L. Heitsch} A cohomology for foliated manifolds. {\it Bull. Amer. Math. Soc.} {\bf 79}, Vol. 6 (1973), 1283--1285.


\bibitem[Hu04]{Huebschmann} {\sc J. Huebschmann} Lie-Rinehart algebras, descent and quantization. {\it Galois theory, Hopf algebras, and semiabelian categories}, 295--316, Fields Inst. Commun., 43, Amer. Math. Soc., Providence, RI, 2004.


\bibitem[Hur]{Hur} {\sc S. Hurder} Characteristic classes for Riemannian foliations.  {\it Differential geometry}, 11--35, World Sci. Publ., Hackensack, NJ, 2009.




\bibitem[IY08]{IY08}{\sc Y. Ilyashenko  and S. Yakovenko} {\it Lectures on analytic differential equations}.
     {Graduate Studies in Mathematics
     86},
  {American Mathematical Society}, 2008.

\bibitem[Le03]{Lee} {\sc J. Lee} {\it Introduction to smooth manifolds.}  Graduate Texts in Mathematics, 218. Springer, 2003.

\bibitem[Ma05]{KCHM} {\sc K. Mackenzie} {\it General Theory of Lie Groupoids and Lie Algebroids.} LMS Lecture Notes, 213. Cambridge University Press, 2005.


\bibitem[Mo88]{Mo88} {\sc P. Molino} {\it Riemannian foliations} Progr. math. 73, Birkh\"auser, Boston, Inc., Boston, MA, 1988.

\bibitem[MoMr03]{MM} {\sc I. Moerdijk and J. Mrc\^{u}n} {\it Introduction to foliations and Lie groupoids.} Cambridge Studies in Advanced Mathematics, 91. Cambridge University Press, 2003.


\bibitem[Pe01]{Pe01} {\sc L. Perko} {\it Differential equations and dynamical systems.} Texts in Applied Mathematics, 7. Springer-Verlag, New York, 2001. 

\bibitem[Po88]{Posi1988} {\sc A. Posilicano} A {L}ie group structure on the space of time-dependent vector fields, {\it Monatsh. Math.} {\bf 105}, no. 4 (1988), 287--293.

\bibitem[PPT]{PPT} {\sc M. J. Pflaum, H. Posthuma, X. Tang} Geometry of orbit spaces of proper Lie groupoids.
 {\it J. Reine Angew. Math.},  to appear. (Available as  arXiv:1101.0180).



\bibitem[St74]{Stefan} {\sc P. Stefan} Accessible sets, orbits, and foliations with singularities. {\it Proc. London Math. Soc.} (3) {\bf 29} (1974), 699--713.

\bibitem[Ste57]{Ste57} {\sc S. Sternberg},
     {Local contractions and a theorem of {P}oincar\'e}.
{\it Amer. J. Math.} {\bf 79} (1957),
 {809--824}.

\bibitem[Su73]{Sussmann}{\sc H. J. Sussmann,} Orbits of families of vector fields and integrability of distributions. {\it Trans. of the A. M. S.} {\bf 180} (1973) 171--188.




\bibitem[We]{We} {\sc A. Weinstein} Linearization of Regular Proper 
Groupoids. {\it J. Inst. Math. Jussieu} {\bf 3} (2002) 493--511.	 

 

	 \end{thebibliography}
\end{document}